\documentclass[12pt]{article}

\usepackage{amssymb}
\usepackage{newtxtext}
\usepackage{amsmath,amsfonts,amsthm,amssymb}
\usepackage[utf8]{inputenc}
\usepackage[T1]{fontenc}    
\usepackage{url}        
\usepackage{geometry}
\usepackage{booktabs}       
\usepackage{amsfonts}       
\usepackage{nicefrac}       
\usepackage{microtype}      
\usepackage{xcolor}         
\usepackage{subfigure}
\usepackage[inline]{enumitem}
\usepackage{cases}
\usepackage{algorithm}
\usepackage{algpseudocode}
\usepackage{graphicx}
\usepackage{dblfloatfix}
\usepackage{float}
\theoremstyle{plain}
\usepackage{makecell}
\usepackage{bm}
\usepackage{setspace}
\usepackage{thmtools}
\usepackage{thm-restate}
\usepackage{cases}
\usepackage{empheq}
\definecolor{dark-gray}{gray}{0.3}
\definecolor{dkgray}{rgb}{.4,.4,.4}
\definecolor{dkblue}{rgb}{0,0,.5}
\definecolor{medblue}{rgb}{0,0,.75}
\definecolor{rust}{rgb}{0.5,0.1,0.1}
\definecolor{darkblue}{rgb}{0,0.08,0.45}
\usepackage[colorlinks]{hyperref}
\usepackage{biblatex}
\hypersetup{urlcolor=rust}
\hypersetup{citecolor=darkblue}
\hypersetup{linkcolor=blue}

\DeclareMathOperator*{\maximize}{maximize}

\DeclareMathOperator*{\argmin}{argmin}

\newtheorem{theorem}{Theorem}
\numberwithin{theorem}{section}
\newtheorem{lemma}{Lemma}
\numberwithin{lemma}{section}

\numberwithin{remark}{section}

\numberwithin{corollary}{section}

\numberwithin{definition}{section}

\numberwithin{assumption}{section}

\declaretheorem[name=Theorem,numberwithin=section]{thm-rest}
\declaretheorem[name=Lemma,numberwithin=section]{lem-rest}
\declaretheorem[name=Corollary,numberwithin=section]{cor-rest}
\declaretheorem[name=Definition,numberwithin=section]{def-rest}
\declaretheorem[name=Proposition,numberwithin=section]{prop-rest}
\declaretheorem[name=Assumption,numberwithin=section]{ass-rest}

\allowdisplaybreaks

\DeclareMathAlphabet{\mathcalorigin}{OMS}{cmsy}{m}{n}

\usepackage[colorinlistoftodos,prependcaption,backgroundcolor=black!5!white,bordercolor=red]{todonotes}
\usepackage[capitalise]{cleveref}
\title{\vspace{-2.0cm}Battery Bidding under Price Uncertainty \\ in 
Wholesale Electricity Markets}
\author{
Vincent Yinjun-Wang\footnote{Stanford MS\&E. Email: yinjunw@stanford.edu} 
\and
Madeleine Udell\footnote{Stanford MS\&E. Email: udell@stanford.edu}
}
\date{}
\begin{document}
\maketitle
\begin{abstract}
Grid-scale batteries increasingly influence outcomes in wholesale electricity markets, but their observed bid patterns remain difficult to interpret. In particular, bids that appear to reflect strategic withholding may instead arise from rational operations under price uncertainty and risk management. We develop an asset-level model of a price-taking battery that submits stepwise buy and sell bid curves in the day-ahead market under a finite set of price scenarios. The battery chooses quantity--price pairs to maximize a mean--CVaR objective subject to physical and market constraints. A direct formulation is a mixed-integer linear program, but we show that its integer decisions can be removed, yielding an exact linear programming reformulation suitable for empirical analysis. Our empirical results deliver three insights. First, withholding behavior can arise even without market power, because scarce stored energy and uncertain future prices increase the value of holding energy. Second, the effect of uncertainty depends on the state of charge: when stored energy is scarce, greater uncertainty raises sell bid prices, whereas when stored energy is abundant it can lower them. Third, risk management reshapes bid curves into layered structures that secure profitable execution across a broad set of scenarios while preserving some exposure to rare but valuable price spikes.
\end{abstract}
\newpage
\section{Introduction}

Deep decarbonization of the electric power sector requires operating the grid reliably with a much larger share of variable renewable generation, and grid-scale batteries are central to this transition. By shifting energy from periods of surplus generation to periods of scarcity, batteries can reduce reliance on fossil ``peaker'' plants, improve reliability, and facilitate renewable integration. But this flexibility is inherently intertemporal: a battery must decide not only how much to charge or discharge now, but also how much to keep stored energy for potentially more valuable future opportunities under substantial price uncertainty.

Battery storage is rapidly evolving from a niche technology to a resource whose bidding decisions can materially affect prices and dispatch outcomes in wholesale electricity markets. In the United States, the number of grid-connected batteries grew roughly $10\times$ from 2021 to 2023, and installed capacity nearly doubled again in 2024 \cite{plumerPopovich2024GiantBatteries}. In the California ISO (CAISO), battery discharge capability rose from roughly 500~MW in 2020 to roughly 13{,}000~MW by December 2024 \cite{caiso_2024_battery_special_report}, equivalent to about 27\% of CAISO's 2024 peak demand of 48{,}323~MW \cite{caiso_2024_battery_special_report}. In 2024, batteries supplied about 8.6\% of CAISO energy during the early-evening peak (5pm--9pm), while charging reached about 14.7\% of system load during solar hours (10am--1pm)\cite{caiso_2024_battery_special_report}. At these penetration levels, battery bidding behavior is no longer negligible in grid operations.

At the same time, observed battery bids in the wholesale electricity markets are difficult to interpret. CAISO's battery reports document persistent gaps between submitted bids and realized day-ahead market (DAM) prices: in 2024, average buy (charge) bids were \$81/MWh below the DAM price, while average sell (discharge) bids were \$230/MWh above it \cite{caiso_2024_battery_special_report}; the corresponding numbers were \$89/MWh and \$145/MWh in 2023 \cite{caiso_2023_battery_special_report}. These patterns resemble \emph{withholding} behavior: bids that systematically avoid being scheduled in the DAM, such as unusually high sell bids or unusually low buy bids. Such patterns naturally raise concern, because if batteries avoid day-ahead commitments and shift activity to the real-time market (RTM), system operators may commit more expensive resources in the DAM, increasing costs and emissions. Yet these bid patterns need not imply market power. CAISO's reports emphasize that existing market power mitigation has had limited impact on battery dispatch \cite{caiso_2023_battery_special_report,caiso_2024_battery_special_report}. This leaves an important interpretive question: are these bid patterns evidence of market power exercise, or can they arise from legitimate operations?\footnote{Such withholding behavior can also arise from the economics of meeting resource adequacy obligations, under which capacity is monetized separately and SoC must be preserved for must-offer/must-perform requirements. Here, we instead ask whether the behavior can arise from economic considerations within the wholesale electricity market itself.}

Answering this question requires understanding the economics of \emph{competitive} battery dispatch. A conventional generator is willing to produce only when the market price exceeds its marginal production cost, which is determined by fuel cost and heat rate \cite{borenstein2002measuring}. For a battery, the analogous marginal cost of discharging is primarily intertemporal: using 1~MWh now depletes stored energy that may be more valuable in a later hour. In principle, that opportunity cost is itself shaped by uncertainty in market fundamentals that drive prices, including load, renewable output, transmission conditions, and generator outages. More importantly, when bids must be submitted before prices are realized, the relevant charge and discharge thresholds need not reflect intertemporal opportunity cost alone. Dispatching energy under market uncertainty exposes batteries to unavoidable risk, which must be priced into their bids as a risk premium. The resulting bid patterns can look like withholding, even when they arise from fully competitive, price-taking behavior. This observation motivates our central question:

\begin{center}
\emph{How should battery operators handle price uncertainty and manage risk in wholesale markets?}
\end{center}

\subsection{Contributions}

\paragraph{Modeling contributions.}
We study battery bidding for energy arbitrage in the DAM, with the goal of explaining how electricity price uncertainty and risk management affect observed bid patterns. Existing bidding models typically optimize hourly charge and discharge quantities over a planning horizon, committing those quantities ex ante rather than choosing the bid prices (as thresholds) at which the battery is willing to buy or sell. Such formulations are not well suited to our question, because the object of interest here is the bid curve itself. We therefore develop a higher-fidelity stylized bidding model with two features. First, the model optimizes a stepwise bid curve, represented as quantity--price pairs\footnote{A sell bid is cleared if the electricity price is greater than or equal to the bid price; similarly, a buy bid is cleared if the electricity price is less than or equal to the bid price.} (MWh and \$/MWh), so that the model matches how batteries participate in practice. Second, uncertainty is represented through a finite set of hourly price scenarios, and the operator's risk preferences are modeled through a mean--CVaR objective that trades off expected revenue against downside risk.

A natural formulation of this problem is a mixed-integer linear program (MILP), because market clearing depends on whether realized prices cross submitted bid prices. We show that, for the sampled-scenario model, these integer decisions can be removed exactly. The key observation is that with finitely many sampled prices, only the relative position of each bid price among those sampled prices matters for clearing. Therefore, bid prices can be restricted to the sampled price set without loss of optimality, and the resulting clearing patterns can be precomputed and embedded as fixed coefficients. This yields an exact linear programming (LP) reformulation.

\paragraph{Operational insights.}
This LP reformulation substantially improves computational tractability, reducing solution times from hours to less than a second in our empirical analysis, and it delivers three main insights. First, withholding behavior can arise even in the absence of market power, because scarce inventory and uncertain future prices rationally increase the value of holding energy. Second, the effect of uncertainty depends on the battery's initial inventory: when energy is scarce, greater uncertainty raises the value of preserving inventory for future high-price hours, whereas with abundant inventory the same uncertainty can instead make the battery more willing to liquidate energy over a broader range of scenarios. Third, risk management reshapes bid curves into layered structures that secure profitable execution in a broad set of scenarios while preserving some exposure to rare but valuable price spikes.

\subsection{Related work}\label{sec:related}
\paragraph{Existing battery modeling.} 
Existing battery bidding models fall into two classes\footnote{See Appendix~\ref{app_existing_work} for mathematical details.}. The first is the deterministic quantity-only formulation, which optimizes hourly charging and discharging quantities against a known, or perfectly forecasted, price path \cite{pozo2023convex}. This model does not optimize the bid prices at which the battery is willing to transact and therefore is not well suited to our research question. The second class extends this quantity-only formulation to stochastic settings via dynamic programming \cite{xu2020operational}. In these models, the value function of stored energy plays a central role: its derivative gives the marginal opportunity value of stored energy and thereby induces threshold-type bid prices. The follow-up paper \cite{qin2024economic} is the closest to our work in spirit, because it also links price uncertainty to withholding behavior in battery bids and shows that, in some settings, uncertainty can rationally increase the marginal opportunity value of stored energy and hence raise sell bid prices---a mechanism that is captured and generalized by our second operational insight. However, this line of work still derives bid prices indirectly from marginal value functions rather than optimizing stepwise bid curves directly. Our model, by contrast, optimizes the bid curve itself under uncertainty and explicitly incorporates risk management. In contrast to \cite{qin2024economic}, we show that uncertainty can rationally decrease sell bid prices when inventory is abundant and increase them when inventory is scarce. Our simulations also capture realistic intraday cycle of midday charging and evening discharging, under which, even without risk management, uncertainty can likewise rationally decrease sell bid prices.

\paragraph{Three most relevant works.}
Our formulation is mainly informed by three related works. First, \cite{mones2023general} develops a stochastic optimization framework for \emph{convergence bidding} that jointly optimizes bid prices and quantities and shows how a scenario-based formulation can remain computationally tractable. Although convergence bidding in financial markets differs from a battery's physical market participation, that paper is conceptually close to ours in treating the bid curve itself, rather than only dispatched quantities, as the object of optimization. In particular, we draw from \cite{mones2023general} the idea that optimal bid prices can be selected from finitely many sampled prices without loss of optimality. Second, \cite{anunrojwong2024battery} studies batteries in a two-settlement market with market power and shows how decentralized, profit-maximizing batteries can distort grid outcomes through quantity withholding, delayed participation from the DAM to the RTM, and reduced RTM responsiveness. That work complements ours: it focuses on strategic price-making behavior and welfare distortions, whereas we study how withholding behavior can arise even for a price-taking battery managing risk under uncertainty. Moreover, \cite{anunrojwong2024battery} acknowledges that its model assumes market fundamentals are probabilistically known, argues that battery behavior is also partly shaped by uncertainty and robustness considerations, and highlights the distinction between strategic behavior and legitimate operating procedures as an important direction for future research. Third, \cite{karaduman2021economics} develops a dynamic stylized equilibrium framework for grid-scale battery in wholesale electricity markets and verifies battery's price impact, equilibrium effects, and the divergence between private and social incentives. That modeling and analysis approach motivates our broader interest in empirically interpreting battery bidding through a stylized model, even though our paper focuses on asset-level bid formation rather than equilibrium market outcomes.

\section{Battery Modeling}\label{sec:model}
Our goal is to understand how much of this behavior can be explained by legitimate behaviors under price uncertainty, as opposed to exercise of market power. We develop a stylized model of  battery energy arbitrage in wholesale electricity markets. We treat the battery as a \emph{price taker}: it takes electricity prices as exogenous and chooses a bidding curve to maximize risk-adjusted expected profit subject to its physical operating constraints. We formalize the model in the DAM context, but the model can be easily used in the RTM.\footnote{In the RTM, the stylized model can be applied directly as a receding-horizon control policy: at each decision time, the operator updates price scenarios, solves the bidding problem over a forward horizon, and implements the first-period decision. Applying the model to DAM participation more literally would require additional co-optimization with ancillary service markets and resource adequacy obligations. We abstract from these features to isolate how price uncertainty and risk management shape battery behavior in wholesale electricity markets.}

\subsection{Conceptual MILP formulation}\label{sec:milp}

We now define the battery's bidding problem formally.

\paragraph{Notation.} $[T] = \{1,\ldots,T\}$.
$\mathbb{I}\{\mathrm{A}\} = 1$ if event $\mathrm{A}$ happens and 0 otherwise.

\paragraph{Stepwise bidding curves.}
For each hour \(t\in[T]\), the battery submits a stepwise \emph{sell} (discharge) and \emph{buy} (charge) curve for the coming next day. Each step \(n\in[N]\) in hour \(t\) is described by a quantity--price pair:
\[
\begin{array}{llll}
\mbox{sell} & x^{\mathrm{sell}}_{t,n} & \mbox{MWh if price is at least} & p^{\mathrm{sell}}_{t,n}, \\
\mbox{buy} & x^{\mathrm{buy}}_{t,n} & \mbox{MWh if price is at most} & p^{\mathrm{buy}}_{t,n}.
\end{array}
\]
Thus \(x^{\mathrm{sell}}_{t,n}\) is the quantity the battery is willing to discharge if the price is higher than \(p^{\mathrm{sell}}_{t,n}\), and \(x^{\mathrm{buy}}_{t,n}\) is the quantity it is willing to charge if the price is lower than \(p^{\mathrm{buy}}_{t,n}\). The battery cannot charge or discharge more than \(\overline{x}\) MWh within a single hour, so
\begin{equation}
\tag{\mbox{power limit}}
\label{eq:power-limit}
\sum_{n\in[N]} x^{\mathrm{sell}}_{t,n}\le \overline{x},
\qquad
\sum_{n\in[N]} x^{\mathrm{buy}}_{t,n}\le \overline{x},
\qquad t\in[T].
\end{equation}
The battery also cannot simultaneously charge and discharge in the same hour. A sufficient way to enforce this is to impose a deadband between the buy and sell bid prices:
\begin{equation}
\label{eq:deadband}
\tag{\mbox{sell over buy}}
p^{\mathrm{sell}}_{t,n}\ge p^{\mathrm{buy}}_{t,m},
\qquad
t\in[T],\; n,m\in[N].
\end{equation}
Under \eqref{eq:deadband}, a realized price cannot simultaneously clear both a buy step and a sell step in the same hour.
For convenience, we write
\[
\mathbf{x}^{\mathrm{sell}}_t=(x^{\mathrm{sell}}_{t,1},\dots,x^{\mathrm{sell}}_{t,N}),
\qquad
\mathbf{x}^{\mathrm{buy}}_t=(x^{\mathrm{buy}}_{t,1},\dots,x^{\mathrm{buy}}_{t,N}),
\]
\[
\mathbf{p}^{\mathrm{sell}}_t=(p^{\mathrm{sell}}_{t,1},\dots,p^{\mathrm{sell}}_{t,N}),
\qquad
\mathbf{p}^{\mathrm{buy}}_t=(p^{\mathrm{buy}}_{t,1},\dots,p^{\mathrm{buy}}_{t,N}).
\]

\paragraph{Price scenarios.}
Uncertainty is represented by a finite scenario set \(\Omega\), with scenario weights \(\pi_\omega>0\) satisfying \(\sum_{\omega\in\Omega}\pi_\omega=1\). Each scenario \(\omega\in\Omega\) specifies a full path of DAM prices \(\lambda^{\mathrm{DA}}_{t,\omega}\) over all hours \(t\in[T]\).

\paragraph{Market clearing.}
The battery submits its bids before prices are realized. After scenario \(\omega\) is realized, a sell bid clears if the market price \(\lambda^{\mathrm{DA}}_{t,\omega}\) exceeds its sell bid price, and a buy bid clears if the market price \(\lambda^{\mathrm{DA}}_{t,\omega}\) falls below its buy bid price. Formally, in scenario \(\omega\),
\[
\mathbb{I}\{\lambda^{\mathrm{DA}}_{t,\omega}\ge p^{\mathrm{sell}}_{t,n}\}\,x^{\mathrm{sell}}_{t,n}
\]
is the cleared discharge quantity of sell step \(n\), while
\[
\mathbb{I}\{\lambda^{\mathrm{DA}}_{t,\omega}\le p^{\mathrm{buy}}_{t,n}\}\,x^{\mathrm{buy}}_{t,n}
\]
is the cleared charge quantity of buy step \(n\).

\paragraph{Expected state-of-charge dynamics.}
The battery cannot charge when it is full or discharge when it is empty. Because prices are uncertain at bid submission time, actual charging and discharging quantities are scenario-dependent. To keep the model tractable while preserving the intertemporal logic of bidding, we impose the battery's inventory constraint through \emph{expected} state-of-charge (SoC) dynamics. Let \(s_t\) denote the expected SoC at the end of hour \(t\), with given initial SoC \(s_0\). Let \(\eta\in(0,1]\) denote the one-way efficiency, and let \(\underline{s}\le s_t\le \overline{s}\) be the storage bounds. The expected SoC evolves as
\begin{align}
s_t
= s_{t-1}
+ \sum_{\omega\in\Omega}\pi_{\omega}\Big(
\eta\sum_{n\in[N]}\mathbb{I}\{\lambda^{\mathrm{DA}}_{t,\omega}\le p^{\mathrm{buy}}_{t,n}\}\,x^{\mathrm{buy}}_{t,n}
-\frac{1}{\eta}\sum_{n\in[N]}\mathbb{I}\{\lambda^{\mathrm{DA}}_{t,\omega}\ge p^{\mathrm{sell}}_{t,n}\}\,x^{\mathrm{sell}}_{t,n}
\Big),
\label{eq:soc_indicator}
\tag{\mbox{soc}}
\end{align}
for each \(t\in[T]\).

\paragraph{Scenario revenue and downside-risk term.}
Given scenario \(\omega\in\Omega\), the battery's revenue equals revenue from cleared sales minus payments for cleared purchases:
\begin{align}
r_\omega
=\sum_{t\in[T]}\lambda^{\mathrm{DA}}_{t,\omega}\Big(
\sum_{n\in[N]}\mathbb{I}\{\lambda^{\mathrm{DA}}_{t,\omega}\ge p^{\mathrm{sell}}_{t,n}\}\,x^{\mathrm{sell}}_{t,n}
-
\sum_{n\in[N]}\mathbb{I}\{\lambda^{\mathrm{DA}}_{t,\omega}\le p^{\mathrm{buy}}_{t,n}\}\,x^{\mathrm{buy}}_{t,n}
\Big).
\label{eq:profit_indicator}
\tag{\mbox{revenue}}
\end{align}

To model risk management, we penalize the conditional value-at-risk (CVaR) of losses\footnote{We call $(-r_\omega)_+$ a loss.} at confidence level \(\alpha\in(0,1)\). Using the standard Rockafellar--Uryasev representation, the CVaR of losses is
\begin{align}
\tau+\frac{1}{1-\alpha}\sum_{\omega\in\Omega}\pi_\omega z_\omega,
\label{eq:cvar}
\tag{\mbox{cvar}}
\end{align}
where \(\tau\in\mathbb{R}\) and \(z_\omega\ge 0\) are CVaR auxiliary variables that satisfy
\begin{align}
z_\omega \ge -r_\omega-\tau,\qquad \omega\in\Omega,
\qquad
z_\omega\ge 0,\qquad \omega\in\Omega.
\label{eq:cvar_auxiliary}
\tag{\mbox{cvar auxiliary}}
\end{align}
We combine expected revenue and downside risk through a balanced objective with weight \(\theta\in[0,1]\). Larger \(\theta\) places more emphasis on expected revenue, while smaller \(\theta\) places more emphasis on limiting downside tail risk.

\paragraph{Conceptual MILP.}
Collecting the above ingredients, we obtain the battery bidding problem:
{\footnotesize
\begin{equation}
\label{eq:milp}
\tag{\mbox{conceptual MILP}}
\begin{array}{lll}
\maximize
& \theta\sum_{\omega\in\Omega}\pi_\omega r_\omega
-(1-\theta)\left(\tau+\frac{1}{1-\alpha}\sum_{\omega\in\Omega}\pi_\omega z_\omega\right) & \\
\textbf{subject to} && \\
\eqref{eq:profit_indicator}
& r_\omega
=\sum_{t\in[T]}\lambda^{\mathrm{DA}}_{t,\omega}\bigg(
\sum_{n\in[N]}\mathbb{I}\{\lambda^{\mathrm{DA}}_{t,\omega}\ge p^{\mathrm{sell}}_{t,n}\}x^{\mathrm{sell}}_{t,n}
-
\sum_{n\in[N]}\mathbb{I}\{\lambda^{\mathrm{DA}}_{t,\omega}\le p^{\mathrm{buy}}_{t,n}\}x^{\mathrm{buy}}_{t,n}
\bigg),
& \omega\in\Omega, \nonumber\\
\eqref{eq:cvar_auxiliary}
& z_\omega \ge -r_\omega-\tau, & \omega\in\Omega, \nonumber\\
& z_\omega\ge 0, & \omega\in\Omega, \nonumber\\
\eqref{eq:soc_indicator}
& s_t=s_{t-1}+\sum_{\omega\in\Omega}\pi_{\omega}\bigg(
\eta\sum_{n\in[N]}\mathbb{I}\{\lambda^{\mathrm{DA}}_{t,\omega}\le p^{\mathrm{buy}}_{t,n}\}x^{\mathrm{buy}}_{t,n}
-\frac{1}{\eta}\sum_{n\in[N]}\mathbb{I}\{\lambda^{\mathrm{DA}}_{t,\omega}\ge p^{\mathrm{sell}}_{t,n}\}x^{\mathrm{sell}}_{t,n}
\bigg), & t\in[T], \nonumber\\
& \underline{s}\le s_t\le \overline{s}, & t\in[T], \nonumber\\
\eqref{eq:power-limit}
& 0\le \sum_{n\in[N]}x^{\mathrm{sell}}_{t,n}\le \overline{x}, & t\in[T], \nonumber\\
& 0\le \sum_{n\in[N]}x^{\mathrm{buy}}_{t,n}\le \overline{x}, & t\in[T], \nonumber\\
\eqref{eq:deadband}
& p^{\mathrm{sell}}_{t,n}\ge p^{\mathrm{buy}}_{t,m}, & t\in[T], \nonumber\\
&& n,m\in[N], \nonumber\\
\textbf{variables} &&\\
\mbox{bid quantities} & (\mathbf{x}^{\mathrm{sell}}_1,\mathbf{x}^{\mathrm{buy}}_1),\dots,(\mathbf{x}^{\mathrm{sell}}_T,\mathbf{x}^{\mathrm{buy}}_T), \\
\mbox{bid prices} &
(\mathbf{p}^{\mathrm{sell}}_1,\mathbf{p}^{\mathrm{buy}}_1),\dots,(\mathbf{p}^{\mathrm{sell}}_T,\mathbf{p}^{\mathrm{buy}}_T), \\
\mbox{SoC dynamics} & s_1,\dots,s_T, \\
\mbox{CVaR variables} & \tau,z_1,\dots,z_{|\Omega|}.
\end{array}
\end{equation}
}

As written, \eqref{eq:milp} is best viewed as a conceptual formulation. The indicator functions \(\mathbb{I}(\cdot)\) hide the binary clearing logic: a standard MILP implementation would introduce binary variables to represent whether each buy or sell step clears in each scenario, based on comparisons between the realized price \(\lambda^{\mathrm{DA}}_{t,\omega}\) and the submitted bid prices \(p^{\mathrm{sell}}_{t,n}\) and \(p^{\mathrm{buy}}_{t,n}\). In the next subsection, we exploit the structure of finite sampled scenarios to eliminate this binary logic exactly and reformulate the model as a LP.

\subsection{Exact LP reformulation}\label{sec:lp}

In this subsection, we show that the clearing indicators can be precomputed, yielding an exact LP reformulation. The key observation is that, once \(x^{\mathrm{sell}}_{t,n}\) or \(x^{\mathrm{buy}}_{t,n}\) is fixed and uncertainty is represented by finitely many sampled prices, the clearing terms
\[
\sum_{n\in[N]}\mathbb{I}\{\lambda^{\mathrm{DA}}_{t,\omega}\ge p^{\mathrm{sell}}_{t,n}\}\,x^{\mathrm{sell}}_{t,n},
\qquad
\sum_{n\in[N]}\mathbb{I}\{\lambda^{\mathrm{DA}}_{t,\omega}\le p^{\mathrm{buy}}_{t,n}\}\,x^{\mathrm{buy}}_{t,n}
\]
are piecewise-constant functions of the bid prices. Therefore, bid prices can be restricted to the sampled prices without changing any clearing outcomes, and the resulting clearing patterns can be embedded as fixed precomputed matrices.

\paragraph{Restricting bid prices to the sampled prices.}
Fix an hour \(t\), and let
\[
\Lambda_t \coloneqq \{\lambda^{\mathrm{DA}}_{t,\omega}:\omega\in\Omega\}
=
\{\lambda^{\mathrm{DA}}_{t,(1)}<\cdots<\lambda^{\mathrm{DA}}_{t,(|\Lambda_t|)}\}
\]
denote the set of distinct sampled prices at hour \(t\), written in increasing order. The following lemma shows that, under finite sampled scenarios, there is no loss in restricting every bid price to \(\Lambda_t\). This finding allows us to fix bid prices and optimize only over bid quantities.

\begin{lemma}\label{prop:grid}
There exists an optimal solution to \eqref{eq:milp} in which
\[
p^{\mathrm{sell}}_{t,n}\in\Lambda_t,
\qquad
p^{\mathrm{buy}}_{t,m}\in\Lambda_t,
\qquad
t\in[T],\; n,m\in[N].
\]
\end{lemma}

\noindent\emph{Proof.}
Fix \(t\) and consider any interval \((\lambda^{\mathrm{DA}}_{t,(k)},\lambda^{\mathrm{DA}}_{t,(k+1)})\) between two consecutive sampled prices. There is no scenario price in this interval, so for every scenario \(\omega\), both indicators
\[
\mathbb{I}\{\lambda^{\mathrm{DA}}_{t,\omega}\ge p\},
\qquad
\mathbb{I}\{\lambda^{\mathrm{DA}}_{t,\omega}\le p\}
\]
remain constant as \(p\) varies within the interval. Hence moving any bid price \(p\) to either endpoint in \(\Lambda_t\) leaves all clearing outcomes unchanged. Since both the objective and all constraints depend on bid prices only through these clearing indicators, such a move leaves feasibility and objective value unchanged. Repeating this argument for every \(t\) and \(n\) yields the claim. \qed

\paragraph{Precomputing clearing patterns.}
For each hour \(t\), choose buy and sell bid price candidates from the sampled prices:
\begin{align*}
\hat{\mathcal P}^{\mathrm{buy}}_t
&=
\{\hat p^{\mathrm{buy}}_{t,(1)}<\cdots<\hat p^{\mathrm{buy}}_{t,(N_t^{\mathrm{buy}})}\}\subseteq \Lambda_t,\\
\hat{\mathcal P}^{\mathrm{sell}}_t
&=
\{\hat p^{\mathrm{sell}}_{t,(1)}<\cdots<\hat p^{\mathrm{sell}}_{t,(N_t^{\mathrm{sell}})}\}\subseteq \Lambda_t,
\end{align*}
with
\[
\hat p^{\mathrm{buy}}_{t,(N_t^{\mathrm{buy}})}<\hat p^{\mathrm{sell}}_{t,(1)}
\]
to preserve the deadband constraint \eqref{eq:deadband}. We interpret
\(x^{\mathrm{buy}}_{t,n}\) and \(x^{\mathrm{sell}}_{t,n}\) as the quantities submitted at these fixed bid (sampled) prices. For each hour \(t\) and scenario \(\omega\), define the corresponding clearing indicator vectors
\begin{align*}
\mathbf b^{\mathrm{buy}}_{t,\omega}
&=
\big(b^{\mathrm{buy}}_{t,\omega,1},\dots,b^{\mathrm{buy}}_{t,\omega,N_t^{\mathrm{buy}}}\big)\in\{0,1\}^{N^{\mathrm{buy}}},
\qquad
b^{\mathrm{buy}}_{t,\omega,n}
=
\mathbb{I}\{\lambda^{\mathrm{DA}}_{t,\omega}\le \hat p^{\mathrm{buy}}_{t,(n)}\},\\
\mathbf b^{\mathrm{sell}}_{t,\omega}
&=
\big(b^{\mathrm{sell}}_{t,\omega,1},\dots,b^{\mathrm{sell}}_{t,\omega,N_t^{\mathrm{sell}}}\big)\in\{0,1\}^{N^{\mathrm{sell}}},
\qquad
b^{\mathrm{sell}}_{t,\omega,n}
=
\mathbb{I}\{\lambda^{\mathrm{DA}}_{t,\omega}\ge \hat p^{\mathrm{sell}}_{t,(n)}\}.
\end{align*}
These vectors are fixed data once the price scenarios have been sampled. With this notation, scenario revenue becomes
\begin{align*}
r_\omega
=
\sum_{t\in[T]}\lambda^{\mathrm{DA}}_{t,\omega}\Big(
\mathbf b^{\mathrm{sell}}_{t,\omega}\cdot \mathbf x^{\mathrm{sell}}_t
-
\mathbf b^{\mathrm{buy}}_{t,\omega}\cdot \mathbf x^{\mathrm{buy}}_t
\Big),
\qquad \omega\in\Omega,
\end{align*}
and the expected SoC dynamics become
\begin{align*}
s_t
=
s_{t-1}
+\sum_{\omega\in\Omega}\pi_\omega\Big(
\eta\,\mathbf b^{\mathrm{buy}}_{t,\omega}\cdot \mathbf x^{\mathrm{buy}}_t
-\frac{1}{\eta}\,\mathbf b^{\mathrm{sell}}_{t,\omega}\cdot \mathbf x^{\mathrm{sell}}_t
\Big),
\qquad t\in[T].
\end{align*}
Hence all indicator logic is embedded into fixed coefficient matrices.

It is useful to distinguish the candidate sizes
\(N^{\mathrm{buy}}\) and \(N^{\mathrm{sell}}\)
from the market-imposed limits on the number of bid segments. The former are modeling choices used to discretize the sampled price support and construct the exact LP reformulation for the sampled problem; the latter are operational constraints of the market platform. In particular, the LP may be solved on a finer candidate price set than the number of segments that will ultimately be submitted. In our simulations, the number of optimized bids with strictly positive quantity rarely exceeds 10. Thus, in practice, the optimized bid curve can usually be submitted directly. When a tighter segment cap is needed, the optimized bids can be compressed ex post by merging nearby.

\paragraph{Exact LP reformulation.}
We can therefore rewrite \eqref{eq:milp} as the following LP:
{\footnotesize
\begin{equation}
\label{eq:lp}
\tag{\mbox{exact LP}}
\begin{array}{llll}
\maximize
& \theta\sum_{\omega\in\Omega}\pi_\omega r_\omega
-\,(1-\theta)\left(\tau+\frac{1}{1-\alpha}\sum_{\omega\in\Omega}\pi_\omega z_\omega\right)
&& \\
\textbf{subject to} & \text{primal constraints} && \text{(dual variables)} \\
\eqref{eq:profit_indicator}
& r_\omega
=
\sum_{t\in[T]}\lambda^{\mathrm{DA}}_{t,\omega}\Big(
\mathbf b^{\mathrm{sell}}_{t,\omega}\cdot \mathbf x^{\mathrm{sell}}_t
-
\mathbf b^{\mathrm{buy}}_{t,\omega}\cdot \mathbf x^{\mathrm{buy}}_t
\Big),
& \omega\in\Omega, & (\kappa_\omega\in\mathbb R), \nonumber\\
\eqref{eq:cvar_auxiliary}
& z_\omega \ge -r_\omega-\tau,
& \omega\in\Omega, & (\gamma_\omega\ge 0), \nonumber\\
& z_\omega \ge 0,
& \omega\in\Omega, & (\delta_\omega\ge 0), \nonumber\\
\eqref{eq:soc_indicator}
& s_t
=
s_{t-1}
+\sum_{\omega\in\Omega}\pi_\omega\Big(
\eta\,\mathbf b^{\mathrm{buy}}_{t,\omega}\cdot \mathbf x^{\mathrm{buy}}_t
-\frac{1}{\eta}\,\mathbf b^{\mathrm{sell}}_{t,\omega}\cdot \mathbf x^{\mathrm{sell}}_t
\Big),
& t\in[T], & (\lambda^{\mathrm{opp}}_t\in\mathbb R), \nonumber\\
& \underline{s}\le s_t\le \overline{s},
& t\in[T], & (\underline{\mu}_t\ge 0,\ \overline{\mu}_t\ge 0), \nonumber\\
\eqref{eq:power-limit}
& \sum_{n\in[N]}x^{\mathrm{sell}}_{t,n}\le \overline{x},
& t\in[T], & (\overline{\nu}^{\mathrm{sell}}_t\ge 0), \nonumber\\
& \sum_{n\in[N]}x^{\mathrm{buy}}_{t,n}\le \overline{x},
& t\in[T], & (\overline{\nu}^{\mathrm{buy}}_t\ge 0), \nonumber\\
& x^{\mathrm{sell}}_{t,n}\ge 0,\quad x^{\mathrm{buy}}_{t,n}\ge 0,
& t\in[T],\; n\in[N], & (\underline{\nu}^{\mathrm{sell}}_{t,n}\ge 0,\ \underline{\nu}^{\mathrm{buy}}_{t,n}\ge 0),\\
\textbf{variables} &&\\
\mbox{bid quantities} & (\mathbf x^{\mathrm{sell}}_1,\mathbf x^{\mathrm{buy}}_1),\dots,(\mathbf x^{\mathrm{sell}}_T,\mathbf x^{\mathrm{buy}}_T), \\
\mbox{SoC dynamics} & s_1,\dots,s_T, \\
\mbox{CVaR variables} & \tau,z_1,\dots,z_{|\Omega|}. 
\end{array}
\end{equation}
}

Note that the sign \(\ge\) of dual variables corresponds the LP formulation where the inequality constraints are \(\le\). \eqref{eq:lp} is exact for the sampled problem, having the same feasible set and optimal value as \eqref{eq:milp}. By Lemma~\ref{prop:grid}, bid prices may be restricted to the sampled prices without loss of optimality. And once bid prices are restricted, every clearing decision is completely determined by the sampled scenarios and can be precomputed.

This reformulation has two important implications. First, it substantially improves computational tractability. With 200 sampled scenarios and a 24-hour horizon, solving \eqref{eq:milp} takes hours in Gurobi, whereas \eqref{eq:lp} solves in seconds. Second, \eqref{eq:lp} exposes economically meaningful dual variables. In particular, the dual variable \(\lambda_t^{\mathrm{opp}}\) on SoC dynamics constraint provides an endogenous measure of the battery's intertemporal opportunity cost. 

\subsection{More practical LP reformulation}\label{sec:practical_lp}

In practice, battery bids are often naturally one-sided at the hourly level. For example, under the CAISO diurnal pattern, there is little reason to submit sell bids during midday solar-trough hours, or buy bids during early evening peaks. The battery's first operational decision is usually whether a given hour is best used to charge, discharge, or remain idle. Once this hourly mode is fixed, the remaining optimization problem is simply to choose the bid prices and quantities in that mode; the market-clearing rule then determines actual dispatch conditional on the realized price. Accordingly, we assign each hour \(t\in[T]\) to one of three modes:
\[
\text{charge},\qquad \text{discharge},\qquad \text{idle}.
\]
Battery operators can tell when low-price charging and high-price discharging hours are likely to occur, from historical prices.

\paragraph{Precomputing clearing patterns.}
Let
\[
a_t^{\mathrm{buy}},\,a_t^{\mathrm{sell}}\in\{0,1\},
\qquad
a_t^{\mathrm{buy}}+a_t^{\mathrm{sell}}\le 1,
\]
where \(a_t^{\mathrm{buy}}=1\) indicates a charge hour, \(a_t^{\mathrm{sell}}=1\) indicates a discharge hour, and \(a_t^{\mathrm{buy}}=a_t^{\mathrm{sell}}=0\) indicates an idle hour. Under this assignment, each hour has only one relevant bid side. Fix hour \(t\), and choose a single bid price set
\[
\hat{\mathcal P}_t=\{\hat p_{t,(1)}<\hat p_{t,(2)}<\cdots<\hat p_{t,(N)}\}\subseteq \Lambda_t.
\]
Let \(\mathbf x_t=(x_{t,1},\dots,x_{t,N})\) denote the corresponding bid quantities. If hour \(t\) is a discharge hour, then \(x_{t,n}\) is the quantity offered to sell at price \(\hat p_{t,(n)}\); if hour \(t\) is a charge hour, then \(x_{t,n}\) is the quantity bid to buy at price \(\hat p_{t,(n)}\); if hour \(t\) is idle, then \(\mathbf x_t=\mathbf 0\).

For each hour \(t\) and scenario \(\omega\), define the precomputed revenue-clearing vector
\[
\mathbf b_{t,\omega}
=
\big(b_{t,\omega,1},\dots,b_{t,\omega,N}\big)\in\mathbb R^N,
\]
with entries
\[
b_{t,\omega,n}
=
a_t^{\mathrm{sell}}\mathbb{I}\{\lambda^{\mathrm{DA}}_{t,\omega}\ge \hat p_{t,(n)}\}
-
a_t^{\mathrm{buy}}\mathbb{I}\{\lambda^{\mathrm{DA}}_{t,\omega}\le \hat p_{t,(n)}\}.
\]
Thus \(\mathbf b_{t,\omega}\) encodes both the hour mode and the clearing rule. In a discharge hour it is a sell-clearing indicator vector, in a charge hour it is minus a buy-clearing indicator vector, and in an idle hour it is identically zero. Similarly, define the SoC-transition vector
\[
\tilde{\mathbf b}_{t,\omega}
=
\big(\tilde b_{t,\omega,1},\dots,\tilde b_{t,\omega,N}\big)\in\mathbb R^N,
\]
with entries
\[
\tilde b_{t,\omega,n}
=
\eta\,a_t^{\mathrm{buy}}\mathbb{I}\{\lambda^{\mathrm{DA}}_{t,\omega}\le \hat p_{t,(n)}\}
-
\frac{1}{\eta}\,a_t^{\mathrm{sell}}\mathbb{I}\{\lambda^{\mathrm{DA}}_{t,\omega}\ge \hat p_{t,(n)}\}.
\]
These vectors are fixed data once the scenario set has been sampled.
With this notation, scenario revenue becomes
\begin{align*}
r_\omega
=
\sum_{t\in[T]}\lambda^{\mathrm{DA}}_{t,\omega}\,
\mathbf b_{t,\omega}\cdot \mathbf x_t,
\qquad \omega\in\Omega,
\end{align*}
and the expected SoC dynamics become
\begin{align*}
s_t
=
s_{t-1}
+\sum_{\omega\in\Omega}\pi_\omega\,
\tilde{\mathbf b}_{t,\omega}\cdot \mathbf x_t,
\qquad t\in[T].
\end{align*}

\paragraph{Practical LP reformulation.}
With the matrices constructed above, \eqref{eq:milp} reduces to the following LP:
{\footnotesize
\begin{equation}
\label{eq:lp2}
\tag{\mbox{practical LP}}
\begin{array}{llll}
\maximize
& \theta\sum_{\omega\in\Omega}\pi_\omega r_\omega
-\,(1-\theta)\left(\tau+\frac{1}{1-\alpha}\sum_{\omega\in\Omega}\pi_\omega z_\omega\right)
&& \\
\textbf{subject to} & \text{primal constraints} && \text{(dual variables)} \\
\eqref{eq:profit_indicator}
& r_\omega
=
\sum_{t\in[T]}\lambda^{\mathrm{DA}}_{t,\omega}\,
\mathbf b_{t,\omega}\cdot \mathbf x_t,
& \omega\in\Omega, & (\kappa_\omega\in\mathbb R), \nonumber\\
\eqref{eq:cvar_auxiliary}
& z_\omega \ge -r_\omega-\tau,
& \omega\in\Omega, & (\gamma_\omega\ge 0), \nonumber\\
& z_\omega \ge 0,
& \omega\in\Omega, & (\delta_\omega\ge 0), \nonumber\\
\eqref{eq:soc_indicator}
& s_t
=
s_{t-1}
+\sum_{\omega\in\Omega}\pi_\omega\,\tilde{\mathbf b}_{t,\omega}\cdot \mathbf x_t,
& t\in[T], & (\lambda^{\mathrm{opp}}_t\in\mathbb R), \nonumber\\
& \underline{s}\le s_t\le \overline{s},
& t\in[T], & (\underline{\mu}_t\ge 0,\ \overline{\mu}_t\ge 0), \nonumber\\
\eqref{eq:power-limit} & \sum_{n\in[N]}x_{t,n}\le \overline{x},
& t\in[T], & (\overline{\nu}_t\ge 0), \nonumber\\
& x_{t,n}\ge 0,
& t\in[T],\; n\in[N], & (\underline{\nu}_{t,n}\ge 0),\\
\textbf{variables} &&\\
\mbox{bid quantities} & \mathbf x_1,\dots,\mathbf x_T, \\
\mbox{SoC dynamics} & s_1,\dots,s_T, \\
\mbox{CVaR variables} & \tau,z_1,\dots,z_{|\Omega|}. 
\end{array}
\end{equation}
}
Relative to \eqref{eq:lp}, \eqref{eq:lp2} removes unnecessary bid variables in hours where the economically relevant action (charge, discharge, idle) is already known, and gets closer to how battery operators bid in practice: first forecast the timing of surplus and scarcity to determine whether an hour should be used for charging, discharging, or idling, then design bid curves within each active hour to handle uncertainty about the level of the realized price.

\section{Theoretical Analysis}

We now turn from formulation to economics. The LP reformulation \eqref{eq:lp} not only makes the bidding problem tractable, but also exposes dual variables with direct economic meaning. Our first result uses the KKT conditions of the LP to characterize when the battery bids its full hourly discharging or charging capability.

\begin{theorem}\label{thm:kkt_bangbang}
Consider a battery operated under \eqref{eq:lp}. Let
\(\lambda^{\mathrm{opp}\star}_t\) denote an optimal dual variable associated with the SoC balance constraint \eqref{eq:soc_indicator} at time \(t\), let \(\gamma^\star_\omega\) denote an optimal dual variable associated with the CVaR auxiliary constraint \eqref{eq:cvar_auxiliary}, and let \(\underline{\mu}_t^\star,\overline{\mu}_t^\star\ge 0\) denote the optimal dual variables associated with the lower and upper SoC bounds \(\underline{s}\le s_t\le \overline{s}\), respectively. Then, for every \(t\in[T]\),
\[
\lambda_t^{\mathrm{opp}\star}
=
\sum_{\tau=t}^T
\bigl(\overline{\mu}_\tau^\star-\underline{\mu}_\tau^\star\bigr).
\]
Moreover, for each hour \(t\in[T]\), the battery bids its full hourly discharging capability, i.e.,
\(\sum_{n\in[N]} x^{\mathrm{sell}}_{t,n}{}^\star=\overline{x}\), whenever
\[
\frac{1}{\eta}\lambda^{\mathrm{opp}}_t{}^\star
<
\max_{n\in[N]}
\left\{
\frac{\sum_{\omega\in\Omega}(\theta\pi_\omega + \gamma^\star_\omega)\,
b^{\mathrm{sell}}_{t,\omega,n}\lambda^{\mathrm{DA}}_{t,\omega}}
{\sum_{\omega\in\Omega}\pi_\omega b^{\mathrm{sell}}_{t,\omega,n}}
\right\}.
\]
Similarly, the battery bids its full hourly charging capability, i.e.,
\(\sum_{n\in[N]} x^{\mathrm{buy}}_{t,n}{}^\star=\overline{x}\), whenever
\[
\eta\lambda^{\mathrm{opp}}_t{}^\star
>
\min_{n\in[N]}
\left\{
\frac{\sum_{\omega\in\Omega}(\theta\pi_\omega + \gamma^\star_\omega)\,
b^{\mathrm{buy}}_{t,\omega,n}\lambda^{\mathrm{DA}}_{t,\omega}}
{\sum_{\omega\in\Omega}\pi_\omega b^{\mathrm{buy}}_{t,\omega,n}}
\right\}.
\]
Finally, the multipliers \(\gamma^\star_\omega\) satisfy
\[
0\le \gamma^\star_\omega \le \frac{1-\theta}{1-\alpha}\pi_\omega,
\qquad \forall \omega\in\Omega,
\]
and
\[
\sum_{\omega\in\Omega}\gamma^\star_\omega = 1-\theta.
\]
In addition, for every scenario \(\omega\) in the lowest \((1-\alpha)\)-tail of revenue,
\[
\gamma^\star_\omega=\frac{1-\theta}{1-\alpha}\pi_\omega.
\]
\end{theorem}

Theorem~\ref{thm:kkt_bangbang} yields an economic interpretation of battery bidding. For a sell bid \(n\) at hour \(t\), define the set of scenarios in which that step clears:
\[
\Omega^{\mathrm{sell}}_{t,n}
=
\{\omega\in\Omega:\lambda^{\mathrm{DA}}_{t,\omega}\ge \hat p^{\mathrm{sell}}_{t,(n)}\}.
\]
Then the term inside the maximization can be rewritten as
\begin{equation}
\frac{\sum_{\omega\in\Omega}(\theta\pi_\omega+\gamma^\star_\omega)
b^{\mathrm{sell}}_{t,\omega,n}\lambda^{\mathrm{DA}}_{t,\omega}}
{\sum_{\omega\in\Omega}\pi_\omega b^{\mathrm{sell}}_{t,\omega,n}}
=
\frac{\sum_{\omega\in\Omega^{\mathrm{sell}}_{t,n}}(\theta\pi_\omega+\gamma^\star_\omega)\lambda^{\mathrm{DA}}_{t,\omega}}
{\sum_{\omega\in\Omega^{\mathrm{sell}}_{t,n}}\pi_\omega}.
\label{eq:cemv_sell}
\tag{\mbox{CEMV-sell}}
\end{equation}
We refer to this object as the conditional expected marginal value to sell (CEMV-sell). It averages prices over exactly those scenarios in which sell bid \(n\) clears, while reweighting scenarios by the risk-adjusted weights \(\theta\pi_\omega+\gamma^\star_\omega\). The denominator,
\[
\sum_{\omega\in\Omega^{\mathrm{sell}}_{t,n}}\pi_\omega,
\]
is the probability under the original scenario measure that sell bid \(n\) clears. The numerator replaces the baseline weights \(\pi_\omega\) by the distorted weights \(\theta\pi_\omega+\gamma^\star_\omega\), so \(\gamma_\omega^\star\) measures the extra weight placed on adverse tail scenarios by the CVaR term in the objective. In the risk-neutral benchmark, where \(\theta=1\) and hence \(\gamma^\star_\omega=0\), this expression reduces to the usual conditional expectation of DAM price under \(\pi_\omega\):
\[
\frac{\sum_{\omega\in\Omega}\pi_\omega
b^{\mathrm{sell}}_{t,\omega,n}\lambda^{\mathrm{DA}}_{t,\omega}}
{\sum_{\omega\in\Omega}\pi_\omega b^{\mathrm{sell}}_{t,\omega,n}}
=
\mathbb{E}\!\left[
\lambda^{\mathrm{DA}}_{t,\omega}
\;\middle|\;
\lambda^{\mathrm{DA}}_{t,\omega}\ge \hat p^{\mathrm{sell}}_{t,(n)}
\right].
\]
The comparison term on the left-hand side,
\begin{equation}
\frac{1}{\eta}\lambda_t^{\mathrm{opp}\star},
\label{eq:emoc_discharge}
\tag{\mbox{EMOC-discharge}}
\end{equation}
is the battery's expected marginal opportunity cost of discharging (EMOC-discharge), i.e., discharging one more unit at hour \(t\) lowers future state of charge and therefore sacrifices future operating opportunities, where, 
\[
\lambda_t^{\mathrm{opp}\star}
=
\sum_{\tau=t}^T
\bigl(\overline{\mu}_\tau^\star-\underline{\mu}_\tau^\star\bigr)
\]
is the cumulative shadow value of all future storage constraints from hour \(t\) onward.  Theorem~\ref{thm:kkt_bangbang} therefore says that the battery bids its full hourly discharging capability whenever at least one sell bid step offers a \eqref{eq:cemv_sell} above this \eqref{eq:emoc_discharge}.

The charging condition has a similar interpretation. Define the corresponding conditional expected marginal cost to buy (CEMC-buy) for each buy bid as 
\begin{equation}
\frac{\sum_{\omega\in\Omega}(\theta\pi_\omega+\gamma^\star_\omega)
b^{\mathrm{buy}}_{t,\omega,n}\lambda^{\mathrm{DA}}_{t,\omega}}
{\sum_{\omega\in\Omega}\pi_\omega b^{\mathrm{buy}}_{t,\omega,n}}
=
\frac{\sum_{\omega\in\Omega^{\mathrm{buy}}_{t,n}}(\theta\pi_\omega+\gamma^\star_\omega)\lambda^{\mathrm{DA}}_{t,\omega}}
{\sum_{\omega\in\Omega^{\mathrm{buy}}_{t,n}}\pi_\omega},
\label{eq:cemc_buy}
\tag{\mbox{CEMC-buy}}
\end{equation}
and the battery's expected marginal opportunity value of charging (EMOV-charge), i.e., the value of carrying one additional unit of stored energy into future hours, as
\begin{equation}
    \eta\lambda_t^{\mathrm{opp}\star}.
    \label{eq:emov_charge}
\tag{\mbox{EMOV-charge}}
\end{equation}
Theorem~\ref{thm:kkt_bangbang} says that the battery bids its full hourly charging capability whenever this \eqref{eq:emov_charge} exceeds the lowest \eqref{eq:cemc_buy} among its buy bids.

\section{Empirical Analysis}

We now bring the theory to data. We study three stylized simulations. The first studies the economics of evening discharging when stored energy is scarce. The second studies the more realistic battery operations of midday charging and evening discharging. The third studies the second simulation under risk management. 

\paragraph{CAISO data and scenario generation.}
We use historical DAM prices at node \texttt{SLATE\_7\_N004} in CAISO over the full 2024 calendar year. Let
\[
\boldsymbol{\lambda}_d=(\lambda_{d,1},\dots,\lambda_{d,24})\in\mathbb{R}^{24}
\]
denote the 24-hour DAM price vector on day \(d\). Across the 366 days in 2024, the hourly standard deviation is reported in Table~\ref{tab:hourly_std_three_rows}.

\begin{table}[htbp]
\centering
\caption{Hourly standard deviations of DAM prices (\$/MWh).}
\label{tab:hourly_std_three_rows}
\begin{tabular}{c|cccccccc}
\toprule
Hour \(t\) 
& 0 & 1 & 2 & 3 & 4 & 5 & 6 & 7 \\
Std. dev. 
& 19.13 & 18.35 & 18.13 & 17.98 & 17.87 & 20.12 & 27.43 & 31.74 \\
\midrule
Hour \(t\) 
& 8 & 9 & 10 & 11 & 12 & 13 & 14 & 15 \\
Std. dev. 
& 32.90 & 30.91 & 30.23 & 30.17 & 30.49 & 30.44 & 31.06 & 32.62 \\
\midrule
Hour \(t\) 
& 16 & 17 & 18 & 19 & 20 & 21 & 22 & 23 \\
Std. dev. 
& 33.02 & 37.41 & 50.98 & 27.35 & 22.09 & 21.33 & 20.33 & 19.70 \\
\bottomrule
\end{tabular}
\end{table}

To model uncertainty around a representative daily price path, we use the 2024 mean trajectory
\[
\bar{\boldsymbol{\lambda}}=(\bar{\lambda}_1,\dots,\bar{\lambda}_{24})\in\mathbb{R}^{24}
\]
as the baseline, and construct stochastic price scenarios using a Gaussian process. Let
\[
r_{d,t}=\lambda_{d,t}-\bar\lambda_t
\]
denote the residual price on day \(d\in[366]\) at hour \(t\in[T]\), and define the empirical residual standard deviation (Table~\ref{tab:hourly_std_three_rows}) at hour \(t\) by
\[
\hat\sigma_t
=
\sqrt{\frac{1}{365}\sum_{d=1}^{366}(r_{d,t}-\bar r_t)^2},
\qquad
\bar r_t=\frac{1}{366}\sum_{d=1}^{366}r_{d,t}.
\]
We also let \(\hat\rho_{tt'}\) denote the \((t,t')\) entry of the empirical correlation matrix of the residuals across days:
\[
\hat\rho_{tt'}
=
\frac{\sum_{d=1}^{366}(r_{d,t}-\bar r_t)(r_{d,t'}-\bar r_{t'})}
{\sqrt{\sum_{d=1}^{366}(r_{d,t}-\bar r_t)^2}\sqrt{\sum_{d=1}^{366}(r_{d,t'}-\bar r_{t'})^2}}.
\]
We model intertemporal dependence across hours using an exponential correlation kernel
\[
R(\beta)_{tt'}=\exp\!\bigl(-\beta|t-t'|\bigr),
\qquad t,t'\in[T],
\]
where \(\beta>0\) is a learnable parameter. We estimate \(\beta\) by least squares:
\[
\hat\beta
=
\argmin_{\beta>0}
\sum_{t\neq t'}
\bigl(\hat\rho_{tt'}-\exp(-\beta|t-t'|)\bigr)^2.
\]
Given \((\bar{\boldsymbol{\lambda}},\hat{\boldsymbol{\sigma}},\hat\beta)\), we can generate \(|\Omega|\) price scenarios for each uncertainty level \(\kappa\ge 0\) from a Gaussian process with covariance kernel
\[
\Sigma(\kappa)_{tt'}
=
\kappa^2 \hat\sigma_t \hat\sigma_{t'} \exp\!\bigl(-\hat\beta|t-t'|\bigr),
\qquad t,t'\in[T].
\]
Let \(D=\mathrm{diag}(\hat\sigma_1,\dots,\hat\sigma_{24})\), then
\[
\Sigma(\kappa)=\kappa^2 D R(\hat\beta) D.
\]
We then sample
\[
\boldsymbol{\lambda}_{\omega}
=
\bar{\boldsymbol{\lambda}} + L(\kappa)\mathbf z_{\omega},
\qquad
\mathbf z_{\omega}\overset{\mathrm{i.i.d.}}{\sim}\mathrm{Gaussian}(\mathbf 0,I_{24}),
\]
where \(L(\kappa)L(\kappa)^\top=\Sigma(\kappa)\). Under this specification, the empirical vector \((\hat\sigma_t)_{t=1}^{24}\) captures the hourly volatility profile, \(\hat\beta\) captures intertemporal dependence, and the scalar \(\kappa\) provides a one-dimensional control over the overall uncertainty level. In our simulations, \(\hat\beta=0.0265\), and the resulting scenario paths are shown in Figure~\ref{2024_scenarios}.

\begin{figure}[htbp]
  \centerline{\includegraphics[width=1\textwidth]{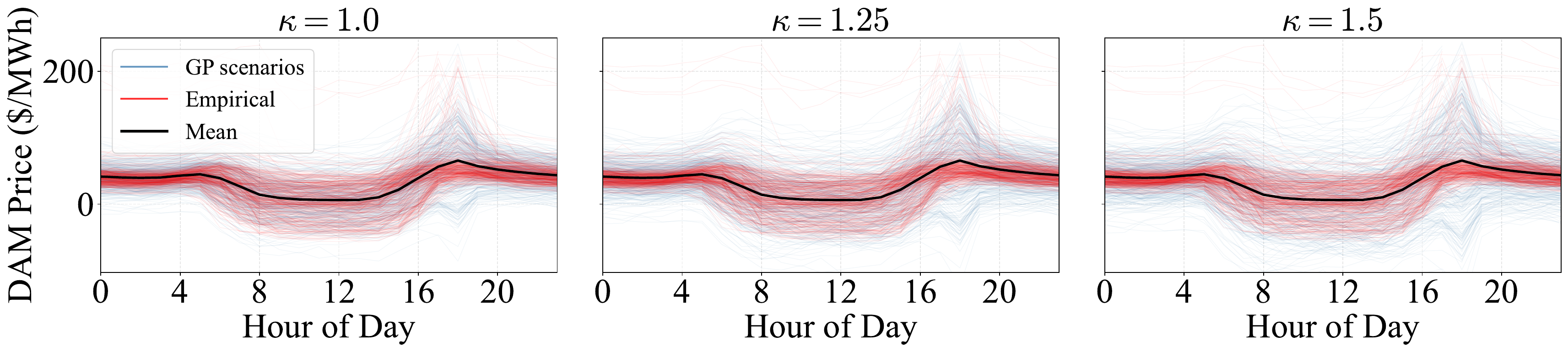}}
  \caption{Generated DAM price scenarios around the 2024 mean trajectory.}
  \label{2024_scenarios}
\end{figure}

\paragraph{Battery configurations.}
We consider a four-hour battery operated under \eqref{eq:lp2} with \(\underline{s}=0\) MWh, \(\overline{s}=32\) MWh, \(\overline{x}=8\) MW, and \(\eta=\sqrt{0.85}\).

\paragraph{Simulation 1: risk-neutral (\(\theta=1\)), early-evening discharging.}
We solve a \(T=6\)  \eqref{eq:lp2} over early evening (4pm--10pm), varying  initial SoC \(\{0,8,16,24,32\}\) at start of discharging hours. 

Figure~\ref{exp1_results} shows that lower starting SoC leads to higher sell bid prices. When the battery begins the evening with less energy, it becomes more selective: it withholds energy from moderate-price hours and reserves discharging capability only for the peaks. Conversely, when the battery begins with abundant energy, the sell bids moves downward, indicating that the operator is willing to monetize inventory over a broader range of prices.

Table~\ref{tab:sigma_s0} quantifies the underlying mechanism through \(\lambda_t^{\mathrm{opp}\star}\). For every fixed uncertainty level \(\kappa\), the \eqref{eq:emoc_discharge} is sharply decreasing with respect to starting SoC: when the battery is inventory-scarce, one additional MWh carried into the evening is much more valuable. 

Table~\ref{tab:sigma_s0} also verifies the effect of uncertainty \(\kappa\) is strongly dependent on starting SoC. For low and intermediate starting SoC, i.e., \(\{8, 16\}\), \(\lambda_t^{\mathrm{opp}\star}\) rises with \(\kappa\), reflecting that greater intraday price volatility increases the option value of waiting for a high-price realization, rather than committing it to the more common prices. This effect is strongest when both uncertainty is high and initial SoC is low, precisely because the \eqref{eq:emoc_discharge} of an early evening hour is then largest. By contrast, when the battery starts the evening with abundant energy, i.e., \(\{24, 32\}\), this effect weakens and can reverse. With a high starting SoC, i.e., \(32\), the battery already has enough inventory to discharge in the most valuable peak hours, and the remaining stored energy can still be liquidated before the horizon ends subject to the hourly \eqref{eq:power-limit} constraints. This excess inventory liquidation pushes optimal sell bid prices downward even under large \(\kappa\).

\begin{table}[htbp] \centering \caption{\(\lambda_t^{\mathrm{opp}\star}\) by uncertainty level \(\kappa\) and starting SoC \(s_{16}\) in simulation 1.} \label{tab:sigma_s0} \begin{tabular}{lcccc} \toprule \(\kappa \backslash s_{16}\) & 8 & 16 & 24 & 32 \\ \midrule 1 & 68.37 & 58.83 & 51.00 & 43.32 \\ 1.25 & 76.38 & 63.32 & 51.69 & 40.88 \\ 1.5 & 84.77 & 67.64 & 52.62 & 38.26 \\ \bottomrule \end{tabular} 
\end{table}

\begin{figure}[htbp]
  \centerline{\includegraphics[width=1\textwidth]{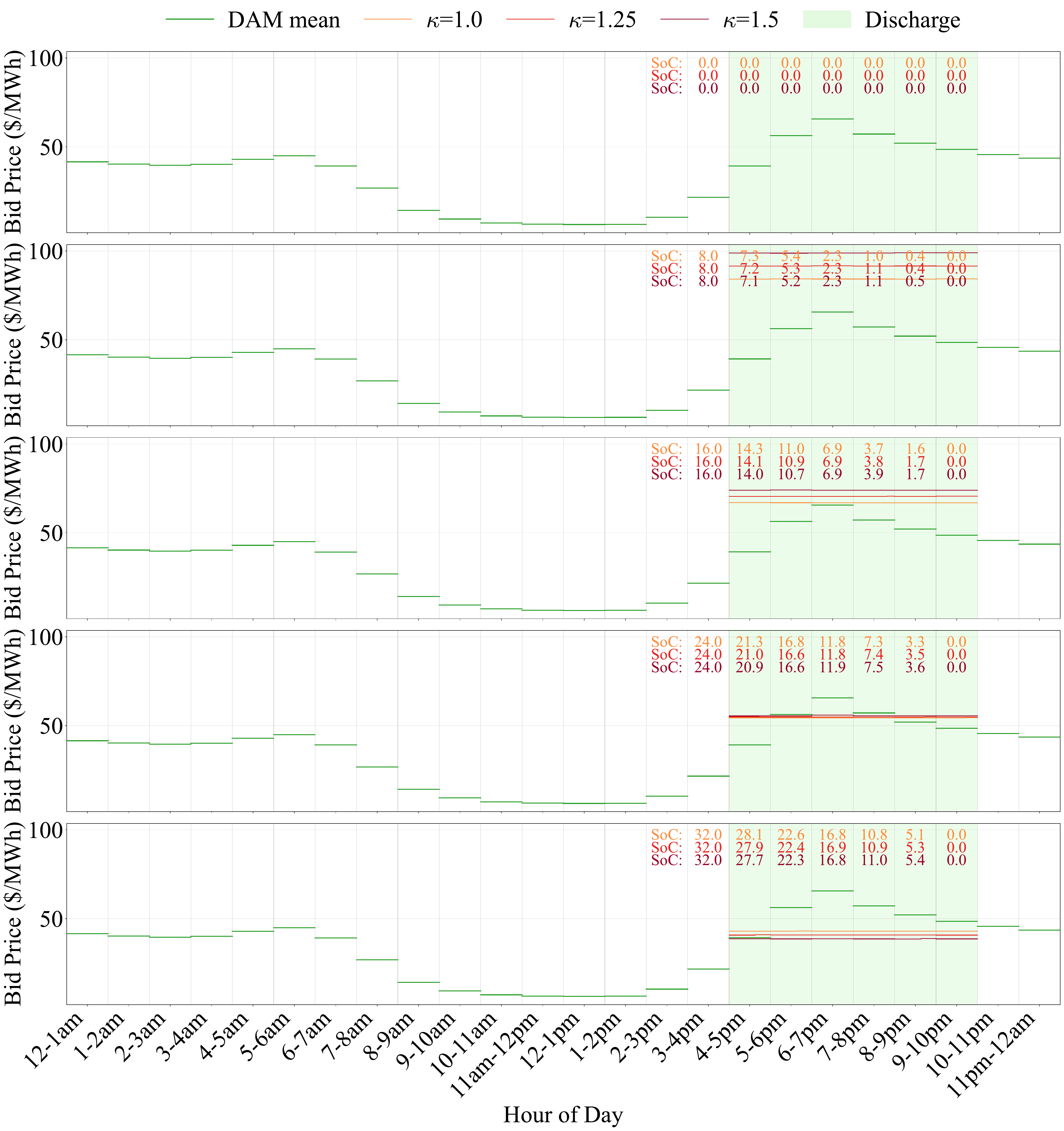}}
  \caption{Risk-neutral discharging simulation.}
  \label{exp1_results}
\end{figure}

\begin{figure}[htbp]
  \centerline{\includegraphics[width=1\textwidth]{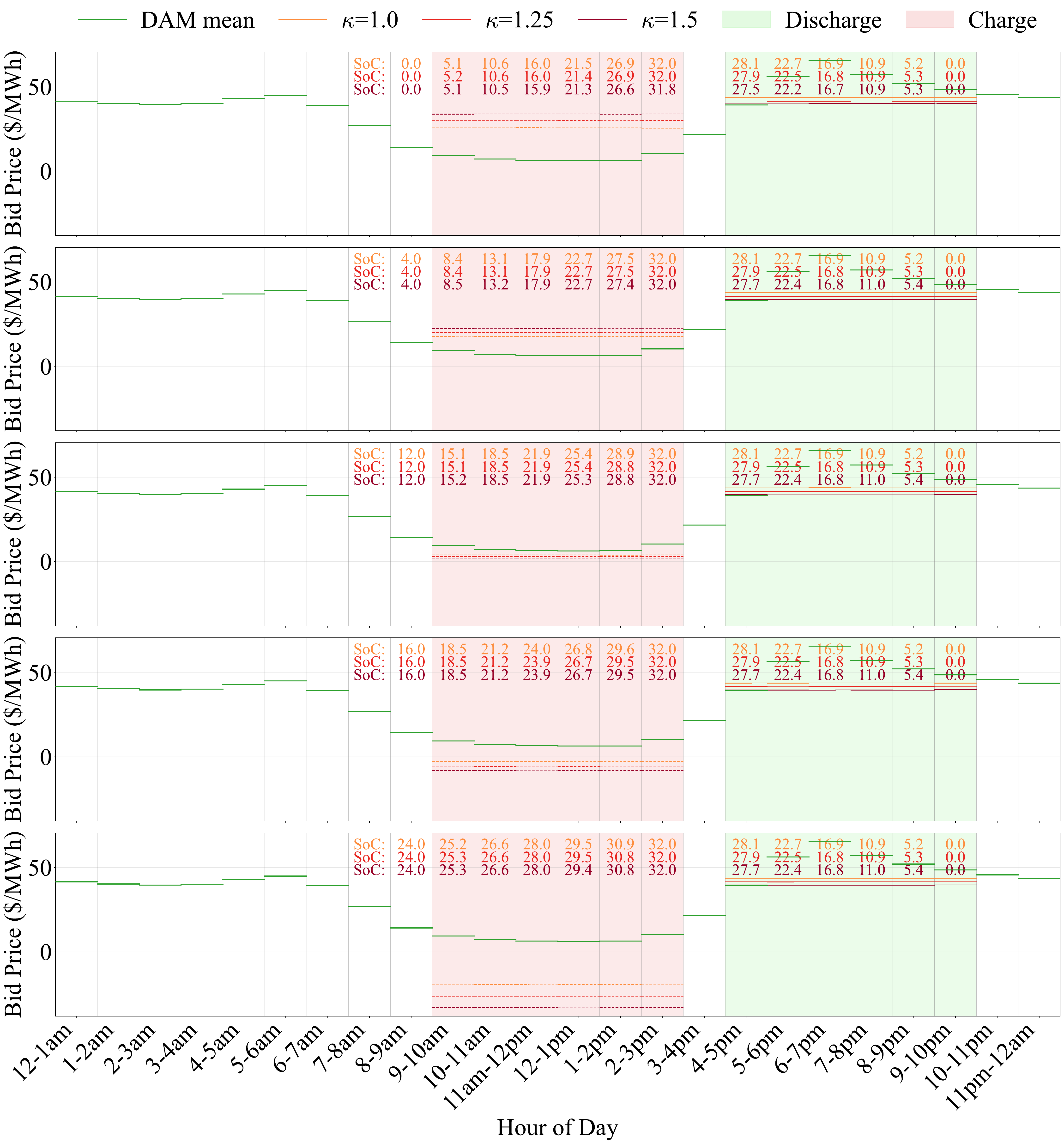}}
  \caption{Risk-neutral simulation.}
  \label{exp2_results}
\end{figure}

\begin{figure}[htbp]
  \centerline{\includegraphics[width=1\textwidth]{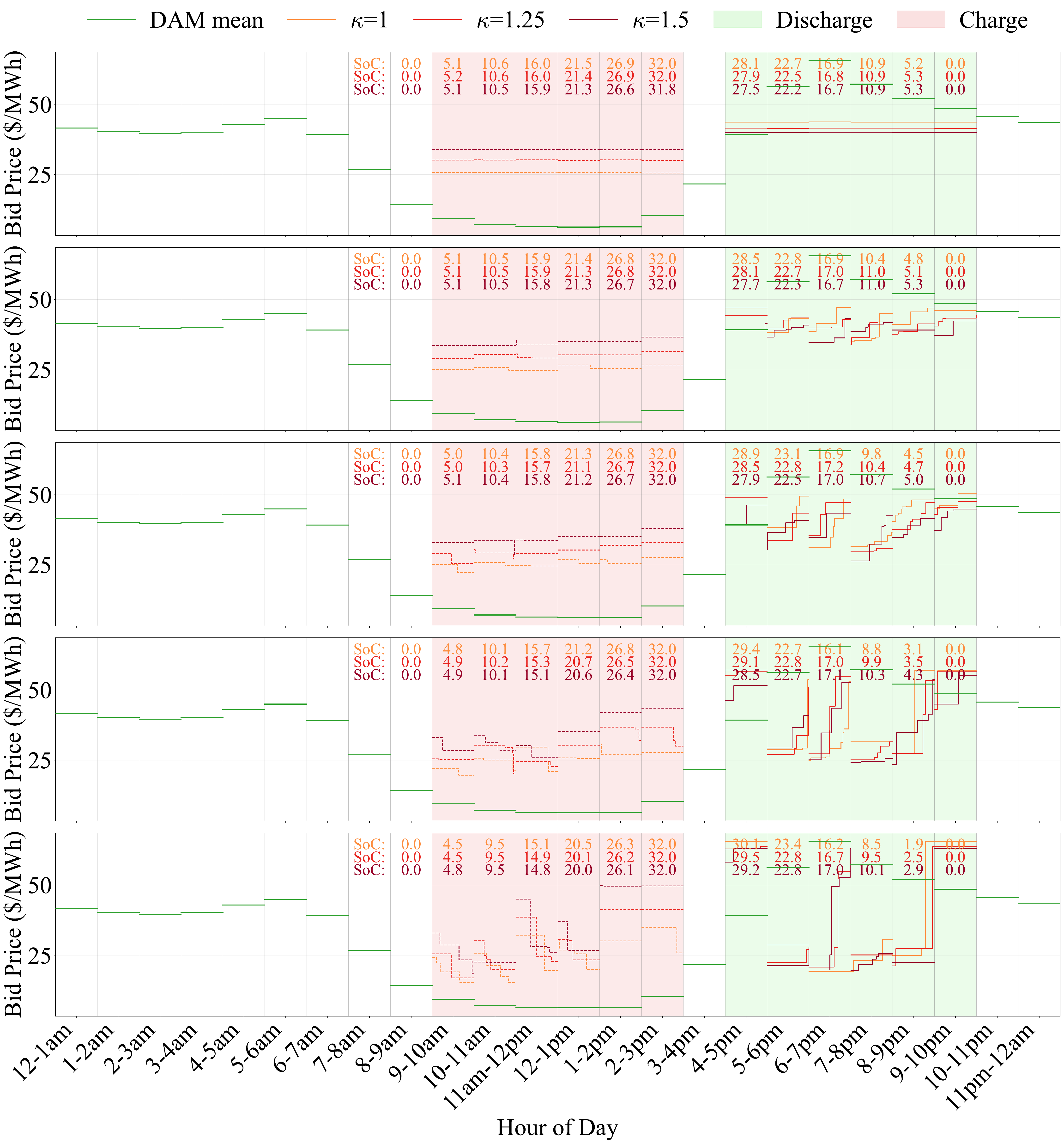}}
  \caption{Risk-averse simulation with \(s_{8}=0\) and \(\theta\in\{1, 0.95, 0.9, 0.8, 0.7\}\) from top to bottom.}
  \label{exp3_0_results}
\end{figure}

\paragraph{Simulation 2: risk-neutral (\(\theta=1\)), midday charging and evening discharging.}
Our second simulation studies a more realistic intraday cycle. We solve a \(T=12\)  \eqref{eq:lp2} in which the battery charges during solar-trough hours (9am--3pm) and discharges during early evening (4pm--10pm), varying the starting SoC \(\{0,4,12,16,24\}\) at the beginning of the charging window.

Across all panels in Figure~\ref{exp2_results}, the battery is effectively driven toward a common full inventory position before the discharge window begins, and then empties before the discharge window ends. When the battery starts the day with little energy, it is willing to buy at substantially higher midday prices in order to rebuild inventory before the evening peak. When it starts with more energy, its buy bid prices fall sharply and can even become negative, indicating that additional charging is attractive only in low-price realizations. By contrast, the evening sell bids are much more stable across the panels because the battery enters the discharge block all from n full SoC position in each case.

Uncertainty reinforces this mechanism, but again depends on the starting SoC. When the battery begins the day with low SoC, i.e., \(\{0,4\}\), a higher \(\kappa\) raises the value of securing inventory for the evening peak, so the battery is willing to charge even at less favorable midday prices. Under these initializations, uncertainty increases the \eqref{eq:emov_charge}, because future high-price discharge opportunities are valuable and inventory is scarce. By contrast, when the battery starts with a high SoC, i.e., \(\{12,16,24\}\), this scarcity largely disappears. Additional charging then becomes a tiny incremental activity rather than a necessity, so higher uncertainty makes the battery more selective on the buy side: it is willing to charge only in deeper trough realizations, and its buy bids move downward, sometimes into negative prices.

\paragraph{Simulation 3: risk-averse, midday charging and evening discharging.}
Our third simulation uses the same \(T=12\) hour setup as Simulation~2, but now introduces downside-risk aversion through the mean--CVaR objective by varying \(\theta\in\{1,0.95,0.9,0.8,0.7\}\), while fixing \(\alpha=0.95\). In simulation 2, we study how uncertainty \(\kappa\) reshapes bidding even under risk neutrality. Here, holding \(\kappa\) fixed, we ask how risk management itself changes bidding behavior.

Risk management trims a small amount of upside in exchange for a substantially stronger downside revenue floor. Across all starting SoC values \(s_8\in\{0,4,8,12,16,24\}\), the battery still charges toward a full inventory position before the evening discharge window and then empties by the end of that window.  Lowering \(\theta\) produces only a modest reduction in expected revenue, but a large improvement in tail performance. For example, when \(s_8=0\) and \(\kappa=1\), lowering \(\theta\) from \(1\) to \(0.7\) reduces expected revenue, \(\sum_{\omega\in\Omega}\pi_\omega r_\omega\), from \(\$2404.55\) to \(\$2287.02\), a drop of less than \(5\%\), while increasing \((1-\alpha)\)-tail revenue, \(-\tau-\frac{1}{1-\alpha}\sum_{\omega\in\Omega}\pi_\omega z_\omega\),
from \(-\$188.14\) to \(\$533.14\). Similar patterns hold across other \((s_8,\kappa)\).

The bid curves (Figures~\ref{exp3_0_results}--\ref{exp3_24_results}) show that risk-managed bidding is not pure de-risking, but rather a rebalancing between core execution and tail opportunity. As \(\theta\) falls, the battery becomes more willing to discharge stored energy across a broader set of price realizations, rather than reserving all discharge for only the most favorable price spikes. This willingness lowers part of the sell bid curve and secures execution in more scenarios, thereby locking in revenue. At the same time, the battery still assigns some quantity to very high sell bid prices in order to preserve exposure to rare but valuable spike realizations. The resulting bid curve is therefore layered: one portion is priced to execute with relatively high probability and stabilize revenue, while another portion remains posted at very high (withholding-looking) prices to retain upside from the right tail of the price distribution. A similar layered structure appears in the charging window.

\section{Conclusion}

This paper studies how wholesale price uncertainty and operators' risk management reshape battery energy arbitrage in wholesale electricity markets. We develop an asset-level LP model in which a price-taking battery chooses stepwise bid curves. The model incorporates uncertainty through finitely many price scenarios and operator's risk preferences through a mean--CVaR objective. Our empirical analysis explains a range of battery behaviors without requiring the exercise of market power. Using CAISO DAM price data, we study how withholding behavior varies with initial SoC, wholesale price uncertainty, and risk management, and show that such patterns need not imply market power. We also characterize how risk management reshapes bid curves into layered structures. These findings point to the need for market monitoring frameworks that distinguish between price-taking, risk-adjusted bidding and true exercises of market power.

An important future direction is to extend this energy arbitrage model to a richer market participation framework. In practice, batteries co-optimize across multiple revenue sources, including energy arbitrage, ancillary services, imbalance reserves, and resource adequacy obligations. These additional opportunities can affect bids through opportunity costs: capacity reserved for ancillary services or reliability obligations cannot always be used for energy arbitrage, and SoC may need to be preserved to satisfy must-offer requirements. Observed bids may also reflect the shadow value of alternative market opportunities and obligations. 

\newpage
\printbibliography

\newpage
\appendix
\section{Appendix}
\subsection{Existing battery modeling}
\label{app_existing_work}
We summarize two representative classes of battery bidding models and explain why they are not sufficient for our purpose.

\paragraph{Deterministic bidding via linear programming.}
We first recall the standard deterministic quantity-only model that underlies much of the battery literature, e.g., \cite{pozo2023convex}. Prices are known and exogenous. Given an initial state of charge (SoC) \(s_0\), the battery chooses hourly charging and discharging quantities,
\[
(x^{\mathrm{sell}}_1,x^{\mathrm{buy}}_1),\ldots,(x^{\mathrm{sell}}_T,x^{\mathrm{buy}}_T),
\]
together with the SoC trajectory \(s_1,\ldots,s_T\), to maximize total revenue:
\begin{equation}
\label{eq:existing_milp}
\tag{\mbox{existing milp}}
\begin{array}{lll}
\maximize
& \sum_{t=1}^T \lambda^{\mathrm{DA}}_t (x^{\mathrm{sell}}_t-x^{\mathrm{buy}}_t) & \\
\textbf{subject to} && \\
\mbox{(soc)} & s_t = s_{t-1} + \eta x^{\mathrm{buy}}_t - \frac{1}{\eta}x^{\mathrm{sell}}_t, & t\in[T], \nonumber\\
& \underline{s} \le s_t \le \overline{s}, & t\in[T], \nonumber\\
\mbox{(power limit)} & 0 \le x_t^{\mathrm{sell}} \le \overline{x}(1-z_t), &  t\in[T],\nonumber\\
& 0 \le x_t^{\mathrm{buy}} \le \overline{x}z_t, & t\in[T],\nonumber \\
\mbox{(binary)} & z_t \in \{0,1\}, &  t\in[T],\nonumber\\
\textbf{variables} &&\\
\mbox{dispatch quantities} & (x^{\mathrm{sell}}_1, x^{\mathrm{buy}}_1),\dots,(x^{\mathrm{sell}}_T, x^{\mathrm{buy}}_T), \\
\mbox{SoC states} & s_1,\dots,s_T, \\
\mbox{mode binaries} & z_{1},\dots,z_{T}. 
\end{array}
\end{equation}
The binary variable \(z_t\) enforces that the battery cannot charge and discharge simultaneously in hour \(t\). 

\paragraph{Stochastic bidding via dynamic programming.}
A second class of models extends the deterministic quantity-only formulation to the stochastic setting and derives optimal quantities, together with implied bid prices, from dynamic programming, e.g., \cite{xu2020operational}. Prices \(\lambda^{\mathrm{DA}}_1,\dots,\lambda^{\mathrm{DA}}_T\) are treated as exogenous random variables with known distribution. The model uses separate charging and discharging variables \(x_t^{\mathrm{buy}},x_t^{\mathrm{sell}}\ge 0\) and replaces the binary restriction on no simultaneous charge/discharge  by the convex relaxation
\[
x_t^{\mathrm{sell}} = 0
\quad \text{whenever } \lambda_t^{\mathrm{DA}}\le 0,
\]
that is,
\[
x_t^{\mathrm{sell}}\le \mathbb{I}\{\lambda_t^{\mathrm{DA}}>0\}\,\overline{x},
\]
which is sufficient to preclude simultaneous charging and discharging (cf.\ Remark~3 in \cite{xu2020operational}). The resulting problem is
\begin{equation}
\label{eq:existing_dp}
\tag{\mbox{existing dp}}
\begin{array}{lll}
\maximize
& \sum_{t=1}^T \lambda^{\mathrm{DA}}_t (x^{\mathrm{sell}}_t-x^{\mathrm{buy}}_t) + V_T(s_T) & \\
\textbf{subject to} && \\
\mbox{(soc)} & s_t = s_{t-1} + \eta x^{\mathrm{buy}}_t - \frac{1}{\eta}x^{\mathrm{sell}}_t, & t\in[T], \nonumber\\
& \underline{s} \le s_t \le \overline{s}, & t\in[T], \nonumber\\
\mbox{(power limit)} & 0 \le x_t^{\mathrm{sell}} \le \mathbb{I}\{\lambda^{\mathrm{DA}}_t \ge 0\}\overline{x}, &  t\in[T],\nonumber\\
& 0 \le x_t^{\mathrm{buy}} \le \overline{x}, & t\in[T],\nonumber \\
\textbf{variables} &&\\
\mbox{dispatch quantities} & (x^{\mathrm{sell}}_1, x^{\mathrm{buy}}_1),\dots,(x^{\mathrm{sell}}_T, x^{\mathrm{buy}}_T), \\
\mbox{SoC states} & s_1,\dots,s_T.
\end{array}
\end{equation}
This model is convex, which makes sensitivity and shadow-price analysis tractable.

Let \(V_t(s)\) denote the expected value function at the beginning of period \(t\) given SoC \(s\in[\underline{s},\overline{s}]\). For each \(t\in[T]\), given incoming SoC \(s_{t-1}\) and realized price \(\lambda_t^{\mathrm{DA}}\), define the feasible action set
\[
\mathcal A_t(s_{t-1},\lambda_t^{\mathrm{DA}})
:=
\left\{
(x_t^{\mathrm{buy}},x_t^{\mathrm{sell}})\in\mathbb R_+^2:
\begin{array}{l}
0\le x_t^{\mathrm{buy}}\le \overline{x},\\[2pt]
0\le x_t^{\mathrm{sell}}\le \mathbb{I}\{\lambda_t^{\mathrm{DA}}>0\}\,\overline{x},\\[2pt]
\underline{s}\le s_{t-1}+\eta x_t^{\mathrm{buy}}-\dfrac{1}{\eta}x_t^{\mathrm{sell}}\le \overline{s}
\end{array}
\right\}.
\]
The post-decision value at time \(t\) is
\begin{align*}
Q_{t-1}(s_{t-1},\lambda_t^{\mathrm{DA}})
&=
\max_{(x_t^{\mathrm{buy}},x_t^{\mathrm{sell}})\in\mathcal A_t(s_{t-1},\lambda_t^{\mathrm{DA}})}
\left\{
\lambda_t^{\mathrm{DA}}(x_t^{\mathrm{sell}}-x_t^{\mathrm{buy}})
+V_t\!\left(s_{t-1}+\eta x_t^{\mathrm{buy}}-\frac{1}{\eta}x_t^{\mathrm{sell}}\right)
\right\},
\end{align*}
and the Bellman recursion is
\begin{align*}
V_{t-1}(s_{t-1})
:=
\mathbb E_{\lambda_t^{\mathrm{DA}}}\!\left[Q_{t-1}(s_{t-1},\lambda_t^{\mathrm{DA}})\right],
\qquad t\in[T],
\end{align*}
with terminal condition \(V_T(s)\). The function \(V_t(\cdot)\) is concave and captures the continuation value of remaining inventory \cite{xu2020operational}. Assuming differentiability for simplicity, let
\[
v_t(s):=\frac{d}{ds}V_t(s).
\]
Then the battery discharges a candidate amount
\[
x\in\Big[0,\min\{\overline{x},\,\eta(s_{t-1}-\underline{s})\}\Big]
\]
only if
\[
\mathrm{realized}\ \lambda_t^{\mathrm{DA}}
\ge
\frac{1}{\eta}\,v_t\!\left(s_{t-1}-\frac{x}{\eta}\right),
\]
and charges a candidate amount
\[
x\in\Big[0,\min\{\overline{x},\,\tfrac{\overline{s}-s_{t-1}}{\eta}\}\Big]
\]
only if
\[
\mathrm{realized}\ \lambda_t^{\mathrm{DA}}
\le
\eta\,v_t\!\left(s_{t-1}+\eta x\right).
\]

\section{Proof of Theorem~\ref{thm:kkt_bangbang}}
\begin{proof}[Proof of Theorem~\ref{thm:kkt_bangbang}]
Because \eqref{eq:lp} is a linear program, strong duality holds and the KKT conditions are necessary and sufficient for optimality. Its Lagrangian is
\begin{align*}
\mathcal L
&=
\theta\sum_{\omega\in\Omega}\pi_\omega r_\omega
-(1-\theta)\left(\tau+\frac{1}{1-\alpha}\sum_{\omega\in\Omega}\pi_\omega z_\omega\right)\\
&\quad
+\sum_{\omega\in\Omega} \kappa_\omega\Bigg(
\sum_{t\in[T]} \lambda^{\mathrm{DA}}_{t,\omega}\Big(
\mathbf b^{\mathrm{sell}}_{t,\omega}{}^\top \mathbf x^{\mathrm{sell}}_t
-\mathbf b^{\mathrm{buy}}_{t,\omega}{}^\top \mathbf x^{\mathrm{buy}}_t
\Big)
- r_\omega
\Bigg)\\
&\quad
-\sum_{\omega\in\Omega} \gamma_\omega\big(-z_\omega-r_\omega-\tau\big)
-\sum_{\omega\in\Omega} \delta_\omega(-z_\omega)\\
&\quad
+\sum_{t\in[T]}\lambda_{t}^{\mathrm{opp}}\!\Bigg(
s_t-s_{t-1}
-\sum_{\omega\in\Omega}\pi_\omega\Big(
\eta\,\mathbf b^{\mathrm{buy}}_{t,\omega}{}^\top \mathbf x^{\mathrm{buy}}_{t}
-\frac{1}{\eta}\,\mathbf b^{\mathrm{sell}}_{t,\omega}{}^\top \mathbf x^{\mathrm{sell}}_{t}
\Big)
\Bigg)\\
&\quad
-\sum_{t\in[T]} \overline{\mu}_t(s_t-\overline s)
-\sum_{t\in[T]} \underline{\mu}_t(\underline s-s_t)\\
&\quad
-\sum_{t\in[T]} \overline{\nu}^{\mathrm{sell}}_{t}\Big(\sum_{n\in[N]}x^{\mathrm{sell}}_{t,n}-\overline x\Big)
+\sum_{t\in[T]}\sum_{n\in[N]}\underline{\nu}^{\mathrm{sell}}_{t,n}x^{\mathrm{sell}}_{t,n}\\
&\quad
-\sum_{t\in[T]} \overline{\nu}^{\mathrm{buy}}_{t}\Big(\sum_{n\in[N]}x^{\mathrm{buy}}_{t,n}-\overline x\Big)
+\sum_{t\in[T]}\sum_{n\in[N]}\underline{\nu}^{\mathrm{buy}}_{t,n}x^{\mathrm{buy}}_{t,n}.
\end{align*}

\paragraph{The CVaR multipliers have the claimed tail structure.}
Stationarity with respect to \(z_\omega\) gives
\[
0=\frac{\partial \mathcal L}{\partial z_\omega}
=
-\frac{1-\theta}{1-\alpha}\pi_\omega+\gamma_\omega^\star+\delta_\omega^\star.
\]
Since \(\delta_\omega^\star\ge 0\), this implies
\[
0\le \gamma_\omega^\star \le \frac{1-\theta}{1-\alpha}\pi_\omega,
\qquad \forall \omega\in\Omega.
\]
Stationarity with respect to \(\tau\) gives
\[
0=\frac{\partial \mathcal L}{\partial \tau}
=-(1-\theta)+\sum_{\omega\in\Omega}\gamma_\omega^\star,
\]
hence
\[
\sum_{\omega\in\Omega}\gamma_\omega^\star = 1-\theta.
\]
Complementary slackness for the constraint \(z_\omega\ge 0\) yields
\[
\delta_\omega^\star z_\omega^\star=0.
\]
Therefore, if \(z_\omega^\star>0\), then \(\delta_\omega^\star=0\), so
\[
\gamma_\omega^\star=\frac{1-\theta}{1-\alpha}\pi_\omega.
\]
Since
\[
z_\omega^\star=\max\{-r_\omega^\star-\tau^\star,0\},
\]
the condition \(z_\omega^\star>0\) is equivalent to
\[
r_\omega^\star<-\tau^\star,
\]
that is, scenario \(\omega\) lies in the lowest \((1-\alpha)\)-tail of the revenue distribution.

\paragraph{Risk-adjusted weights.}
Stationarity with respect to \(r_\omega\) gives
\[
0=\frac{\partial \mathcal L}{\partial r_\omega}
=\theta\pi_\omega-\kappa_\omega^\star+\gamma_\omega^\star,
\]
so
\[
\kappa_\omega^\star=\theta\pi_\omega+\gamma_\omega^\star,
\qquad \forall \omega\in\Omega.
\]

\paragraph{Opportunity-cost recursion.}
Stationarity with respect to \(s_t\) gives, for \(t=1,\dots,T-1\),
\[
0=\frac{\partial \mathcal L}{\partial s_t}
=
\lambda_t^{\mathrm{opp}\star}-\lambda_{t+1}^{\mathrm{opp}\star}
-\overline{\mu}_t^\star+\underline{\mu}_t^\star,
\]
and for the terminal hour \(T\),
\[
0=\frac{\partial \mathcal L}{\partial s_T}
=
\lambda_T^{\mathrm{opp}\star}
-\overline{\mu}_T^\star+\underline{\mu}_T^\star.
\]
Hence
\[
\lambda_t^{\mathrm{opp}\star}
=
\lambda_{t+1}^{\mathrm{opp}\star}
+\overline{\mu}_t^\star-\underline{\mu}_t^\star,
\qquad t=1,\dots,T-1,
\]
and
\[
\lambda_T^{\mathrm{opp}\star}
=
\overline{\mu}_T^\star-\underline{\mu}_T^\star.
\]
Backward substitution yields
\[
\lambda_t^{\mathrm{opp}\star}
=
\sum_{\tau=t}^T
\bigl(\overline{\mu}_\tau^\star-\underline{\mu}_\tau^\star\bigr),
\qquad t\in[T].
\]

\paragraph{Stationarity for a sell step.}
Fix \((t,n)\). Stationarity with respect to \(x^{\mathrm{sell}}_{t,n}\) gives
\begin{align*}
0=\frac{\partial\mathcal L}{\partial x^{\mathrm{sell}}_{t,n}}
&=
\sum_{\omega\in\Omega}\kappa_\omega^\star b^{\mathrm{sell}}_{t,\omega,n}\lambda^{\mathrm{DA}}_{t,\omega}
-\frac{1}{\eta}\lambda_t^{\mathrm{opp}}{}^\star\sum_{\omega\in\Omega}\pi_\omega b^{\mathrm{sell}}_{t,\omega,n}
+\underline{\nu}^{\mathrm{sell}}_{t,n}{}^\star
-\overline{\nu}^{\mathrm{sell}}_{t}{}^\star.
\end{align*}
Since \(\sum_{\omega\in\Omega}\pi_\omega b^{\mathrm{sell}}_{t,\omega,n}>0\) and
\(\kappa_\omega^\star=\theta\pi_\omega+\gamma_\omega^\star\), this is equivalent to
\begin{align*}
\frac{1}{\eta}\lambda^{\mathrm{opp}}_t{}^\star
&=
\frac{\sum_{\omega\in\Omega}(\theta\pi_\omega+\gamma_\omega^\star)
b^{\mathrm{sell}}_{t,\omega,n}\lambda^{\mathrm{DA}}_{t,\omega}}
{\sum_{\omega\in\Omega}\pi_\omega b^{\mathrm{sell}}_{t,\omega,n}}
+
\frac{\underline{\nu}^{\mathrm{sell}}_{t,n}{}^\star-\overline{\nu}^{\mathrm{sell}}_{t}{}^\star}
{\sum_{\omega\in\Omega}\pi_\omega b^{\mathrm{sell}}_{t,\omega,n}}.
\end{align*}
Now suppose \(\sum_{n\in[N]}x^{\mathrm{sell}}_{t,n}{}^\star<\overline{x}\). Then complementary slackness implies
\[
\overline{\nu}^{\mathrm{sell}}_{t}{}^\star
\Big(\sum_{n\in[N]}x^{\mathrm{sell}}_{t,n}{}^\star-\overline x\Big)=0
\quad\Longrightarrow\quad
\overline{\nu}^{\mathrm{sell}}_{t}{}^\star=0.
\]
Using dual feasibility \(\underline{\nu}^{\mathrm{sell}}_{t,n}{}^\star\ge 0\), we obtain, for every \(n\in[N]\),
\[
\frac{1}{\eta}\lambda^{\mathrm{opp}}_t{}^\star
\ge
\frac{\sum_{\omega\in\Omega}(\theta\pi_\omega+\gamma_\omega^\star)
b^{\mathrm{sell}}_{t,\omega,n}\lambda^{\mathrm{DA}}_{t,\omega}}
{\sum_{\omega\in\Omega}\pi_\omega b^{\mathrm{sell}}_{t,\omega,n}}.
\]
Therefore, by contrapositive,
\[
\frac{1}{\eta}\lambda^{\mathrm{opp}}_t{}^\star
<
\max_{n\in[N]}
\left\{
\frac{\sum_{\omega\in\Omega}(\theta\pi_\omega+\gamma_\omega^\star)
b^{\mathrm{sell}}_{t,\omega,n}\lambda^{\mathrm{DA}}_{t,\omega}}
{\sum_{\omega\in\Omega}\pi_\omega b^{\mathrm{sell}}_{t,\omega,n}}
\right\}
\;\Longrightarrow\;
\sum_{n\in[N]}x^{\mathrm{sell}}_{t,n}{}^\star=\overline{x}.
\]

\paragraph{Stationarity for a buy step.}
Fix \((t,n)\). Stationarity with respect to \(x^{\mathrm{buy}}_{t,n}\) gives
\begin{align*}
0=\frac{\partial\mathcal L}{\partial x^{\mathrm{buy}}_{t,n}}
&=
-\sum_{\omega\in\Omega}\kappa_\omega^\star b^{\mathrm{buy}}_{t,\omega,n}\lambda^{\mathrm{DA}}_{t,\omega}
+\eta\,\lambda_t^{\mathrm{opp}}{}^\star\sum_{\omega\in\Omega}\pi_\omega b^{\mathrm{buy}}_{t,\omega,n}
+\underline{\nu}^{\mathrm{buy}}_{t,n}{}^\star
-\overline{\nu}^{\mathrm{buy}}_{t}{}^\star.
\end{align*}
Equivalently,
\begin{align*}
\eta\,\lambda^{\mathrm{opp}}_t{}^\star
&=
\frac{\sum_{\omega\in\Omega}(\theta\pi_\omega+\gamma_\omega^\star)
b^{\mathrm{buy}}_{t,\omega,n}\lambda^{\mathrm{DA}}_{t,\omega}}
{\sum_{\omega\in\Omega}\pi_\omega b^{\mathrm{buy}}_{t,\omega,n}}
+
\frac{\overline{\nu}^{\mathrm{buy}}_{t}{}^\star-\underline{\nu}^{\mathrm{buy}}_{t,n}{}^\star}
{\sum_{\omega\in\Omega}\pi_\omega b^{\mathrm{buy}}_{t,\omega,n}}.
\end{align*}
If \(\sum_{n\in[N]}x^{\mathrm{buy}}_{t,n}{}^\star<\overline{x}\), then complementary slackness implies
\[
\overline{\nu}^{\mathrm{buy}}_{t}{}^\star=0.
\]
Using \(\underline{\nu}^{\mathrm{buy}}_{t,n}{}^\star\ge 0\), we obtain for every \(n\in[N]\),
\[
\eta\lambda^{\mathrm{opp}}_t{}^\star
\le
\frac{\sum_{\omega\in\Omega}(\theta\pi_\omega+\gamma_\omega^\star)
b^{\mathrm{buy}}_{t,\omega,n}\lambda^{\mathrm{DA}}_{t,\omega}}
{\sum_{\omega\in\Omega}\pi_\omega b^{\mathrm{buy}}_{t,\omega,n}}.
\]
Therefore, by contrapositive,
\[
\eta\lambda^{\mathrm{opp}}_t{}^\star
>
\min_{n\in[N]}
\left\{
\frac{\sum_{\omega\in\Omega}(\theta\pi_\omega+\gamma_\omega^\star)
b^{\mathrm{buy}}_{t,\omega,n}\lambda^{\mathrm{DA}}_{t,\omega}}
{\sum_{\omega\in\Omega}\pi_\omega b^{\mathrm{buy}}_{t,\omega,n}}
\right\}
\;\Longrightarrow\;
\sum_{n\in[N]}x^{\mathrm{buy}}_{t,n}{}^\star=\overline{x}.
\]
\end{proof}

\section{Empirical Results}

\begin{figure}[htbp]
  \centerline{\includegraphics[width=1\textwidth]{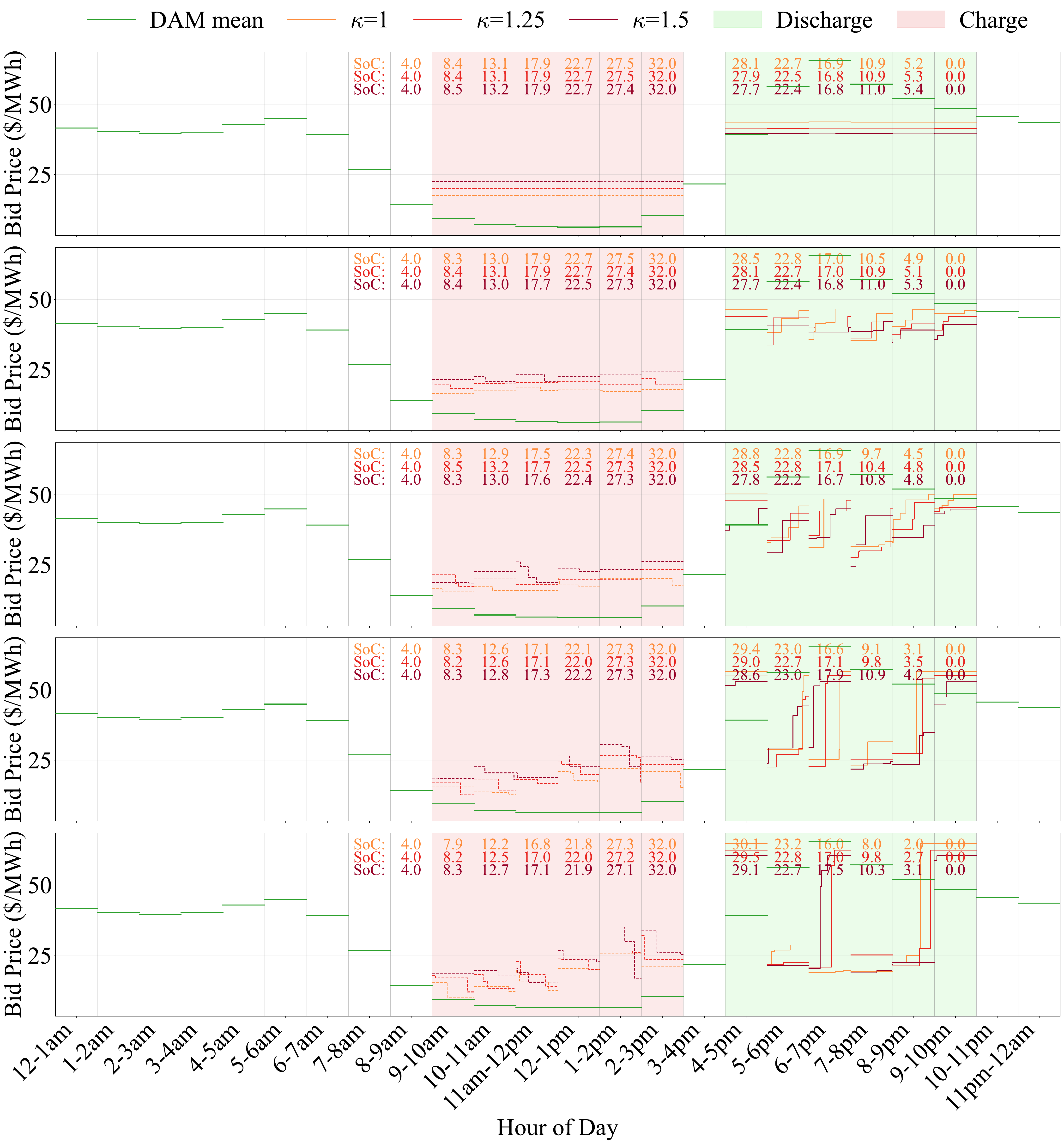}}
  \caption{Risk-aware simulation with \(s_{8}=4\) and \(\theta\in\{1, 0.95, 0.9, 0.8, 0.7\}\) from top to bottom.}
  \label{exp3_4_results}
\end{figure}

\begin{figure}[htbp]
  \centerline{\includegraphics[width=1\textwidth]{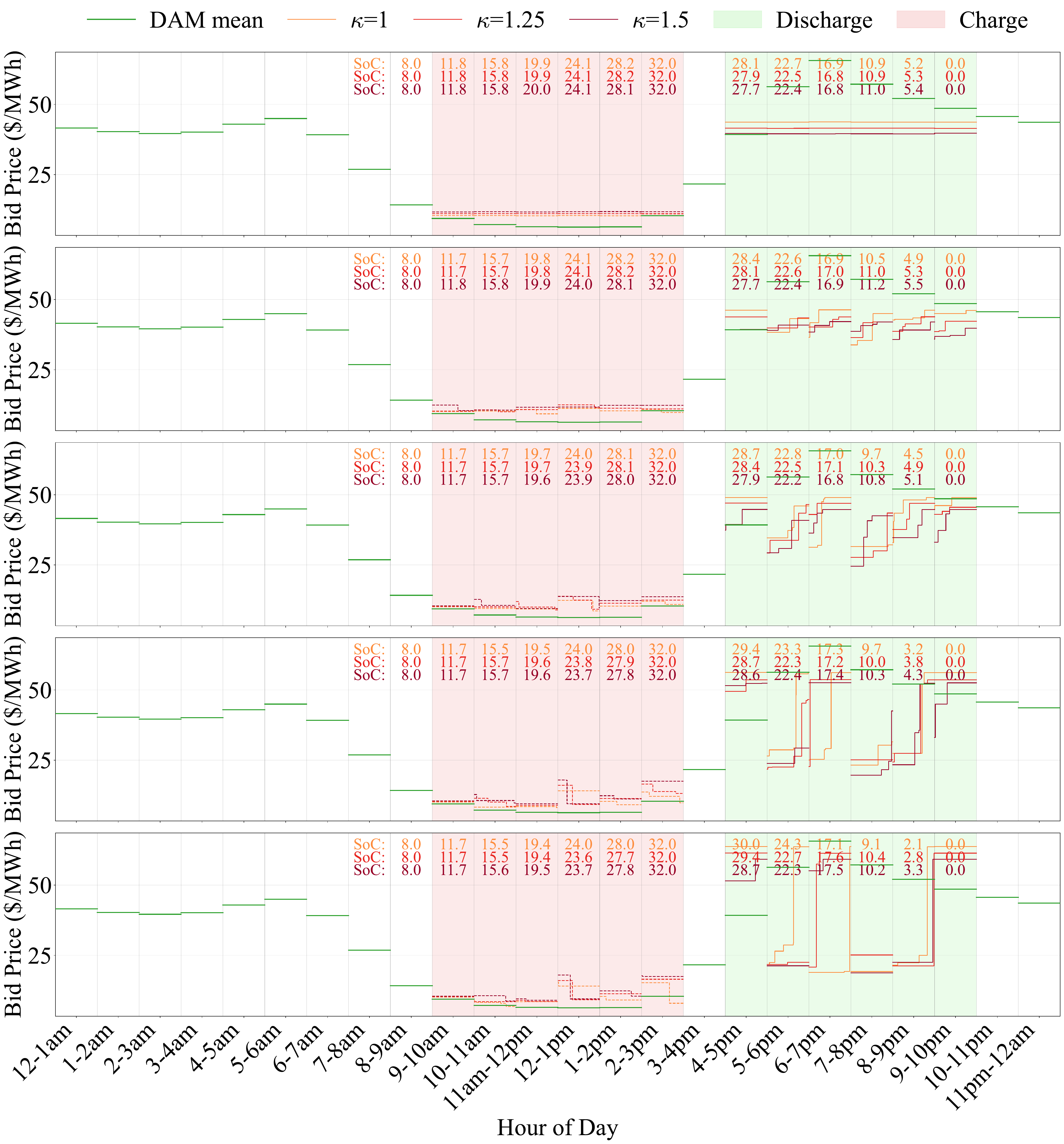}}
  \caption{Risk-Aware Simulation \(s_{8}=8\) and \(\theta\in\{1, 0.95, 0.9, 0.8, 0.7\}\) from top to bottom.}
  \label{exp3_8_results}
\end{figure}

\begin{figure}[htbp]
  \centerline{\includegraphics[width=1\textwidth]{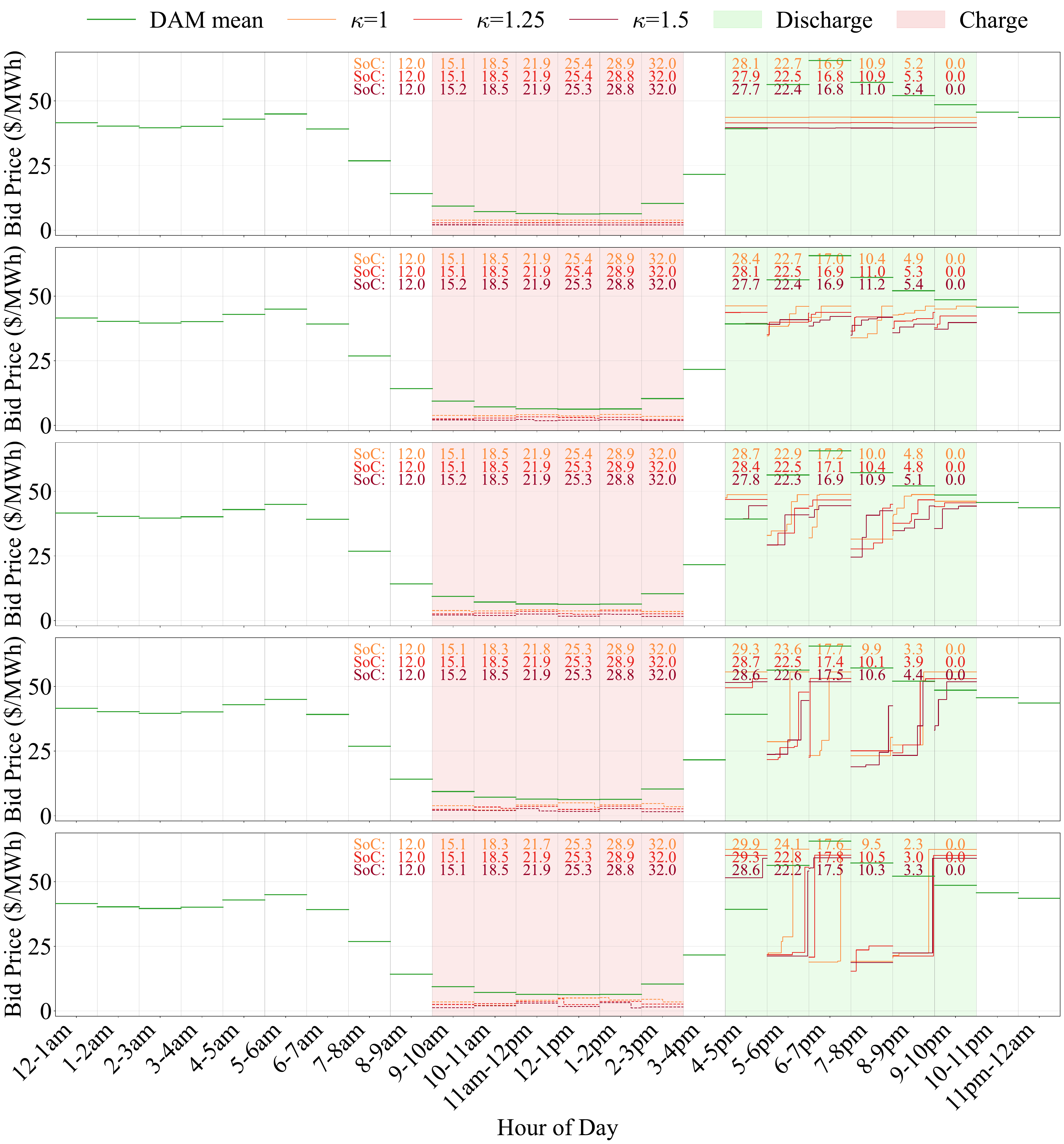}}
  \caption{Risk-Aware Simulation \(s_{8}=12\) and \(\theta\in\{1, 0.95, 0.9, 0.8, 0.7\}\) from top to bottom.}
  \label{exp3_12_results}
\end{figure}

\begin{figure}[htbp]
  \centerline{\includegraphics[width=1\textwidth]{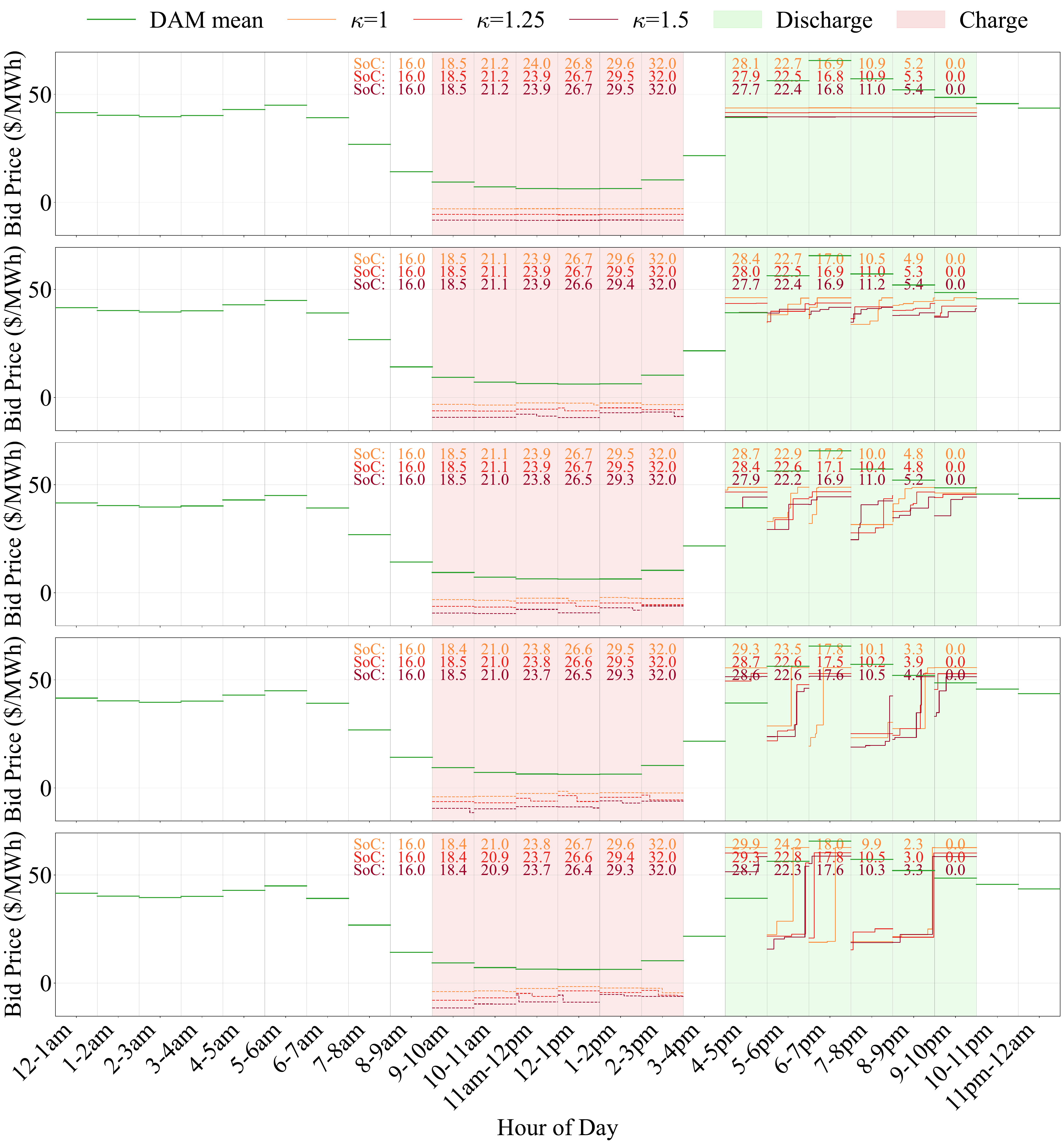}}
  \caption{Risk-Aware Simulation \(s_{8}=16\) and \(\theta\in\{1, 0.95, 0.9, 0.8, 0.7\}\) from top to bottom.}
  \label{exp3_16_results}
\end{figure}

\begin{figure}[htbp]
  \centerline{\includegraphics[width=1\textwidth]{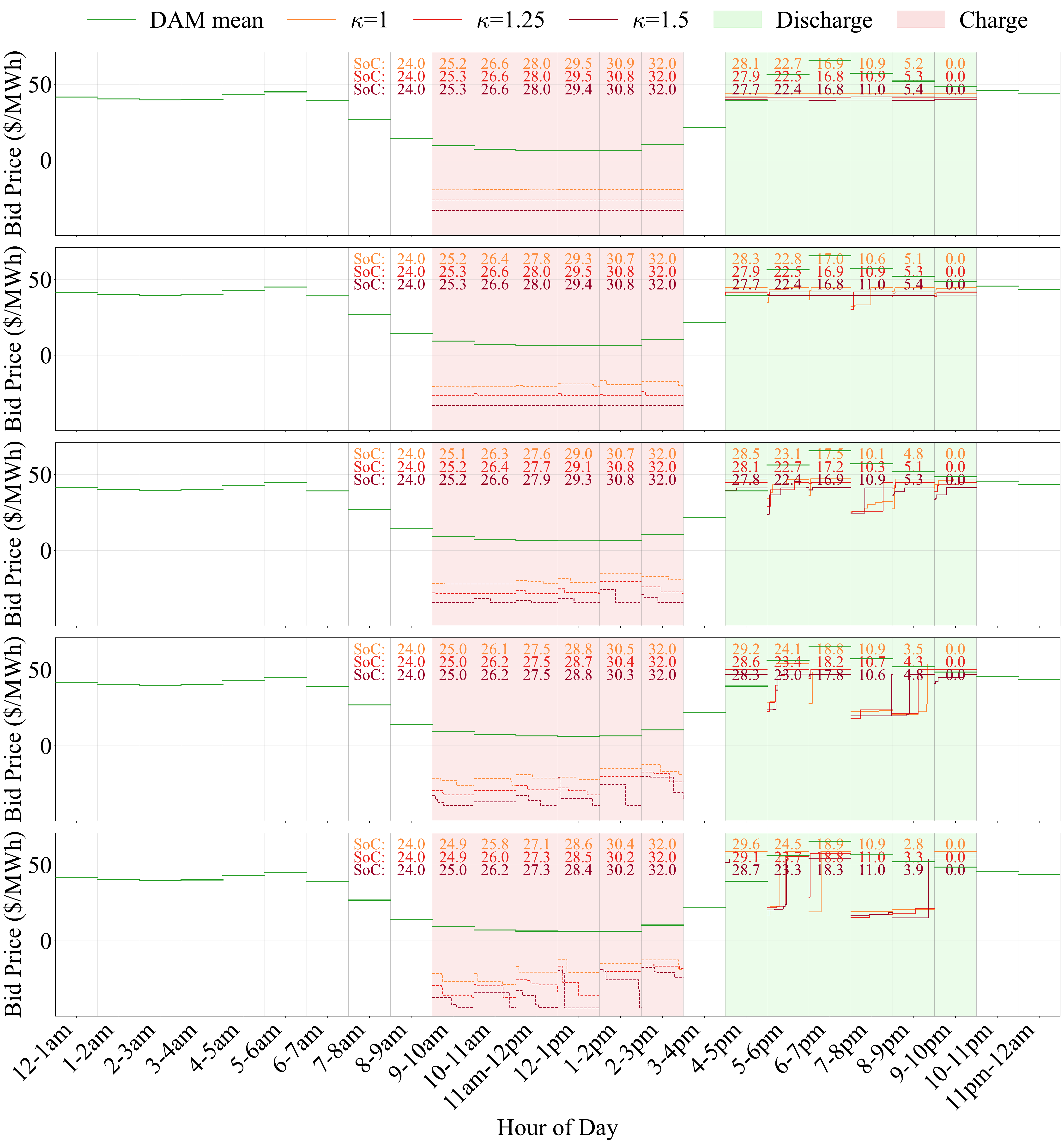}}
  \caption{Risk-Aware Simulation \(s_{8}=24\) and \(\theta\in\{1, 0.95, 0.9, 0.8, 0.7\}\) from top to bottom.}
  \label{exp3_24_results}
\end{figure}

\end{document}